\documentclass[12pt]{amsart}
\usepackage[utf8]{inputenc}
\usepackage{amsfonts}
\usepackage{amsthm}
\usepackage{amsmath}
\usepackage{amssymb}
\usepackage{fullpage}
\usepackage{mathtools}
\usepackage{tikz}
\usepackage{enumerate} 
\usepackage{float}

\usepackage[hidelinks]{hyperref}

\newtheorem{teo}{Theorem}[section]

\newtheorem{cor}[teo]{Corollary}
\newtheorem{lem}[teo]{Lemma}
\newtheorem{rem}[teo]{Remark}
\newtheorem{prop}[teo]{Proposition}

\newtheorem{ex}[teo]{Example}
\newtheorem{question}[teo]{Question}
\newcommand{\ignore}[1]{}

\newcommand{\Hom}{\operatorname{\hom}}
\newcommand{\N}{\mathbb{N}}

\newcommand{\Rec}{\mathcal{R}}

\begin{document}
\title{Minimal reconstructions of a coloring}
\author{Diego Gamboa}
\address{Escuela de Matem\'aticas, Universidad Industrial de Santander, C.P. 680001, Bucaramanga - Colombia.}
\email{dgambh@gmail.com}
\author{Carlos Uzc\'{a}tegui-Aylwin}
\address{Escuela de Matem\'aticas, Universidad Industrial de Santander, C.P. 680001, Bucaramanga - Colombia.}
\email{cuzcatea@saber.uis.edu.co}
\date{}

\begin{abstract}
A coloring on  a finite or countable set $X$ is a function $\varphi: [X]^{2} \to \{0,1\}$, where $[X]^{2}$ is the collection of unordered pairs of $X$. The collection  of homogeneous sets for $\varphi$, denoted  by $\Hom(\varphi)$, consist of all $H \subseteq X$ such that $\varphi$ is constant on $[H]^2$; clearly, $\Hom(\varphi) = \Hom(1-\varphi)$. A coloring  $\varphi$ is \textit{reconstructible} up to complementation from its homogeneous sets if, for any coloring $\psi$ on $X$ such that $\Hom(\varphi) = \Hom(\psi)$, either $\psi = \varphi$ or $\psi = 1-\varphi$. By $\mathcal{R}$ we denote the collection of all colorings reconstructible from their homogeneous sets.  Let $\varphi$ and $\psi$ be colorings on $X$, and set
\[
D(\varphi, \psi) = \{ \{x,y\} \in [X]^2: \; \psi\{x,y\} \neq \varphi\{x,y\}\}.
\]
If $\varphi\not\in \mathcal{R}$, let
\[
r(\varphi) = \min\{|D(\varphi, \psi)|: \; \Hom(\varphi) = \Hom(\psi), \, \psi \neq \varphi, \, \psi \neq 1-\varphi\}.
\]
A coloring $\psi$ such that $\Hom(\varphi)=\Hom(\psi)$,  $\varphi\neq \psi$ and $1-\varphi\neq \psi$ is called a  {\em non trivial reconstruction} of $\varphi$. If, in addition, $r(\varphi) =|D(\varphi, \psi)|$, we call $\psi$  a {\em minimal reconstruction} of $\varphi$. 
The purpose of this article is to study the minimal reconstructions of a coloring. 
We show that, for large enough $X$, $r(\varphi)$ can only takes the values $1$ or  $4$. 
\end{abstract}


\maketitle

\section{Introduction}

A coloring on  a finite or countable set $X$ is a function $\varphi: [X]^{2} \to \{0,1\}$, where $[X]^{2}$ is the collection of unordered pairs of $X$. The collection  of homogeneous sets for $\varphi$, denoted  by $\Hom(\varphi)$, consist of all $H \subseteq X$ such that $\varphi$ is constant on $[H]^2$; clearly, $\Hom(\varphi) = \Hom(1-\varphi)$. We say that $\varphi$ is \textit{reconstructible} up to complementation from its homogeneous sets if, for any coloring $\psi$ on $X$ such that $\Hom(\varphi) = \Hom(\psi)$, either $\psi = \varphi$ or $\psi = 1-\varphi$.  In the terminology of graphs, we are discussing graphs that can be reconstructed (up to complementation) from the collection of their cliques and independent sets.
In this paper we present some new results about  this  reconstruction problem, which  has been already analyzed in \cite{Dammak2019,  Pouzet2013,clari-uzca2023,Pouzetetal2011,Pouzetetal2016}.
The motivation of this line of research comes from the well known  Ulam’s reconstruction conjecture, still unsolved (see the survey \cite{bondy1991}).

A typical condition for being  a non-reconstructible coloring is the presence of a pair $\{x, y\}$ such that $\varphi(\{x,z\}) = 1-\varphi(\{y,z\})$ for all $z \in X \setminus \{x, y\}$. Such pairs are called \textit{critical}. Indeed, suppose $\{x,y\}$ is a critical pair for $\varphi$ and let $\psi$ be the coloring defined by $\psi\{a,b\}=\varphi\{a,b\}$ for all $\{a,b\}\neq \{x,y\}$ and  $\psi\{x,y\}=1-\varphi\{x,y\}$, then $\Hom(\psi)=\Hom(\varphi)$ and thus $\varphi$ is not reconstructible. 

We denote by $\mathcal{R}$ the collection of all colorings reconstructible up to complementation from their homogeneous sets and by $\neg \mathcal{R}$ those which are not reconstructible. 
Let $\varphi$ and $\psi$ be colorings on $X$, and set
\[
D(\varphi, \psi) = \{ \{x,y\} \in [X]^2: \; \psi\{x,y\} \neq \varphi\{x,y\}\}.
\]
Consider the function $r:\neg\mathcal{R} \to \mathbb{N} \cup \{\infty\}$ defined in \cite{clari-uzca2023} by
\[
r(\varphi) = \min\{|D(\varphi, \psi)|: \; \Hom(\varphi) = \Hom(\psi), \, \psi \neq \varphi, \, \psi \neq 1-\varphi\}.
\]
A coloring $\psi$ such that $\Hom(\varphi)=\Hom(\psi)$,  $\varphi\neq \psi$ and $1-\varphi\neq \psi$ is called a  {\em non trivial reconstruction} of $\varphi$. If, in addition, $r(\varphi) =|D(\varphi, \psi)|$, we call $\psi$  a {\em minimal reconstruction} of $\varphi$. 
The main purpose of this article is to study the minimal reconstructions of a coloring. 

We know that $r(\varphi) = 1$ if and only if there is a critical pair for $\varphi$ and, moreover, $r(\varphi) \neq 2$ for all $\varphi \not\in \mathcal{R}$ (see \cite{clari-uzca2023}). One of our results is that, for $X$ large enough,  there are only two possibilities for $r(\varphi)$: Either $r(\varphi) = 1$ or $r(\varphi) = 4$ (see theorem \ref{teo:SumasMinimales}). We also isolate the typical coloring for which $r$ takes value 4, we call it a {\em critical cycle}.

A very simple but crucial observation for our arguments is that $\{x,y\} \in D(\varphi, \psi)$ if and only if $(\varphi + \psi)(\{x,y\}) = 1$, where $\varphi + \psi$ is the Boolean sum (addition mod 2). The idea of using Boolean sums comes from \cite{Dammak2019, Pouzet2013, Pouzetetal2011}. They showed that if two graphs $G$ and $G'$ (over an infinite set) have the same 3-element homogeneous sets, then the components of $G+G'$ (or of its complement) are either paths or cycles of even length \cite[Theorem 1.2]{Pouzetetal2011}. We present an alternative proof of their result (see Theorem \ref{teo:Pouzetetal}). The connection of these ideas with the possible values of the function $r$ is as follows. Suppose $\psi$ is a minimal reconstruction of $\varphi$, then $D_1(\varphi+\psi)=\{\{x,y\}\in [X]^2:\; (\varphi+\psi)(\{x,y\}=1\}$ is connected (see Lemma \ref{recon-comp}).
Thus we concentrate in studying  the connected components of  $D_1(\varphi+\psi)$.
We use the structure of the components to present a very precise description of the restriction of a coloring $\varphi$ to a component of $D_1(\varphi+\psi)$ where $\psi$ is any non-trivial reconstruction of $\varphi$. 

On the other hand, a sufficient condition for a coloring $\varphi$ to be reconstructible is that for all $F\subseteq X$ with $|F|= 4$ there are  $x,y\in X\setminus F$ such that  $\varphi\{x,z\}=\varphi\{y,z\}=\varphi\{x,y\}$ for all $z\in F\setminus \{x,y\}$ (see \cite[Prop. 3.5]{clari-uzca2023}). This type of colorings also satisfy that for every finite set $F$ there is another finite set $G\supseteq F$ such that the restriction of $\varphi$ to $[G]^2$ is  reconstructible. Thus, the natural question left in \cite{clari-uzca2023} was whether this is true for any coloring in $\mathcal{R}$; that is,  suppose  $\varphi\in \mathcal{R}$ and $F\subseteq X$ is a finite set, is there a finite set $G\supseteq F$ such that $\varphi\restriction [G]^2\in \mathcal{R}$? In the last section of the paper we answer negatively this question and present a new finitistic characterization of $\mathcal{R}$.

\ignore{ 
For completeness, in Section \ref{sec:main} we will give an alternative proof of this theorem taking advantage of the hypothesis of $X$ being countably infinite, the proof in \cite{Pouzetetal2011} is more general. This theorem gives more context and paints a more complete picture of what the main result of this article says.
}

\ignore{
Another result of this work that also answers (negatively) a question, in \cite{clari-uzca2023}, regarding a characterization of the property $\mathcal{R}$ for colorings on $\mathbb{N}$ in terms of their \textit{initial segments} and if the property $\mathcal{R}$ holds for them or not. The initial segments of a coloring are the restrictions $\{\varphi\upharpoonright [\{0,1,\dots, n-1\}]^2\}_{n\in\mathbb{N}}$. In \cite{clari-uzca2023}, it was shown that
\begin{prop}\label{prop:FRec}
    Let $\varphi\in 2^{[\mathbb{N}]^2}$. If there are infinite initial segments $\varphi\upharpoonright [\{0,1,\dots,n-1\}]^2$ for which property $\mathcal{R}$ holds then $\mathcal{R}$ holds for $\varphi$.
\end{prop}
    
The question was if the converse of Proposition \ref{prop:FRec} is also true. A construction of a counterexample is given, a coloring $\varphi$ such that $\varphi\upharpoonright [\{0,1,\dots,n-1\}]^2\in \neg \mathcal{R}$ for all $n\in\mathbb{N}$ is defined and it is shown that $\varphi \in \mathcal{R}$.
}

\section{Preliminaries}\label{sec:prelim}

We will use standard notation from Set Theory. Throughout the article, $X$ will denote an at most countable set. Given $k\in\N$,  $[X]^{k}$ denotes  the collection of all subsets of $X$ of size $k$. The notations $[X]^{\leq k}$ and $[X]^{\geq k}$ have the corresponding meanings. The collection of all mappings from $[X]^2$ into $\{0,1\}$ is denoted by $2^{[X]^2}$, and an element from this collection is called a {\em coloring (on $X$)}. Given $Y\subseteq X$ and $\varphi\in 2^{[X]^2}$, $\varphi\restriction Y$ denotes the restriction of $\varphi$ to $[Y]^2$. We make use of the following notation to refer to the subset of $[X]^2$ characterized by $\varphi$.
\begin{equation}
\begin{alignedat}{1}
    D_1(\varphi) &= \{\{x,y\}\in [X]^2\,:\, \varphi\{x,y\}=1\}.
\end{alignedat}    
\end{equation}

There is a straightforward correspondence between colorings and graphs. Indeed, given a graph $G=(X,E)$ and a coloring $\varphi:[X]^2\to 2$ they are regarded as being the same when for every $\{x,y\}\in [X]^2$
\[
\{x,y\} \in E \iff \varphi\{x,y\}=1.
\]
In other words, $E=D_1(\varphi)$.

We borrow some notation from graph theory.  Given a subset $Y$ of $X$ and $A \subseteq [X]^2$, we say that $G=(Y, [Y]^2 \cap A)$ is the \textit{graph induced by $Y$ on $A$} or that the set $Y$ {\em induces the graph $G$ on $A$}. Note that if $\varphi$ is a coloring on $X$, then the graph induced by some $Y\subseteq X$ on $D_1(\varphi)$ is $(Y, D_1(\varphi)\cap [Y]^2) = (Y, D_1(\varphi \restriction [Y]^2))$.

A \textit{path of length $n$ on a subset $A$ of $[X]^2$} is a subset $P=\{x_0, \ldots, x_n\}$ of $X$ that induces a graph $(P,E)$ with edge set
\begin{equation}
\label{eq:pathdef}
       E= \{\{x_k,x_{k+1}\}:k\in \{0,\ldots, n-1\}\} = [P]^2 \cap A  \neq \emptyset.
\end{equation}
For a coloring  $\varphi:[X]^2\to 2$, a {\em path for $\varphi$} is any path induced by some $Y\subseteq X$ on $D_i(\varphi)$ for some $i\in\{0,1\}$.

 A \textit{cycle of length $n+1$ on a subset $A$ of $[X]^2$} is a subset $Z=\{x_0,\ldots, x_n\}$ of $X$ that induces a graph $(Z,E)$ with edge set
\begin{equation}
\label{eq:cycle1}
E =\{\{x_k,x_{k+1}\}: k\in\{0,\dots, n-1\}\} \cup \{\{x_{n},x_0\}\} = [Z]^2 \cap A \neq \emptyset.
\end{equation}
A {\em cycle for $\varphi$} is any cycle induced by some $Y\subseteq X$ on $D_i(\varphi)$ for some $i\in \{0,1\}$.

The \textit{components of ($X$ on) $A\subseteq [X]^2$} are the usual connected components of the graph $(X,A)$, more precisely, the equivalence classes of the relation on $X$ given by
\begin{equation*}
    x \sim_{A}  y \iff\, \text{ there is a path from $x$ to $y$ in $A$.}
\end{equation*}
Note that any component has at least two elements, since otherwise there would be no path as defined above.
Given a coloring $\varphi$ we will often consider the components of $(X,D_i(\varphi+\psi)$ for some $i\in\{0,1\}$ (mostly $i=1$).

A {\em claw in $A$} is a graph $F=(V, E)$ where $V=\{x,y,z,w\}$ and
\begin{equation}\label{eq:claw}
    E = [\{x,y,z\}]^2 = [V]^2 \cap A
\end{equation}
(see figure \ref{fig:claw}). A \textit{claw for $\varphi$} is any claw induced by some $F\in [X]^4$ on $D_i(\varphi+\psi)$ for some $i\in \{0,1\}$.

\begin{figure}[h]
    \centering
    \begin{tikzpicture}[scale=1.5]
\node[above] (w) at (0,0) {\footnotesize$w$};
\node[below] (z) at (1,-1) {\footnotesize$z$};

\node[below] (y) at (0,-1) {\footnotesize$y$};
\node[below] (x) at (-1, -1) {\footnotesize$x$};

\bgroup
\tikzstyle{every node}=[circle, draw, fill=black,inner sep=0pt, minimum width=2.5pt]
\node (w) at (0,0) {};
\node (z) at (1,-1) {};

\node (y) at (0,-1) {};
\node (x) at (-1, -1) {};

\draw[thick] (x) -- (w) -- (y);
\draw[thick] (w) -- (z);
\draw[gray!50,thick] (x) -- (y) -- (z) to[out=-135, in=-45] (x);
\egroup

\end{tikzpicture}
    \caption{A claw.}
    \label{fig:claw}
\end{figure}

For each $A\subseteq [X]^2$ the degree of $x\in X$ with respect to $A$ is defined as follows:
\[
\deg_A(x)=\left|\{y\in X\setminus \{x\}\,:\, \{x,y\} \in A\}\right|.
\]

Let $\varphi:[X]^2\to 2$ be a coloring on $X$, we consider also the case   $X$  finite. The collection of  \textit{homogeneous sets for} $\varphi$, denoted $\Hom(\varphi)$, is the following:
\begin{equation}\label{eq:Hom}
    \Hom(\varphi)=\{H\in [X]^{\geq 3}\,:\, \varphi \upharpoonright [H]^2 =\{i\}, \text{ for some } i \in \{0,1\}\}.
\end{equation}
The well known Ramsey's theorem says that $\Hom(\varphi)$ is not empty whenever  $|X|\geq 6$. 
Additionally, we find it useful to distinguish homogeneous sets by their color
\begin{equation}\label{eq:Homi}
 \begin{alignedat}{2}
     \Hom_0(\varphi)&=\{ H\in \Hom(\varphi)\,:\, \varphi\upharpoonright [H]^2=0\}\\
     \Hom_1(\varphi)&=\{ H\in \Hom(\varphi)\,:\, \varphi\upharpoonright [H]^2=1\}
 \end{alignedat}    
\end{equation}  

We say that $\varphi,\psi\in 2^{[X]^2}$ are \textit{reconstructions of the other} if $\Hom(\varphi)=\Hom(\psi)$. Clearly $\varphi$ and $1-\varphi$ are reconstruction of $\varphi$, they are called {\em trivial reconstructions}.
The main results of this paper are about the equivalence classes of the relation $H$ defined by 
\begin{equation}
\begin{alignedat}{1}
    \varphi H \psi &\iff \Hom(\varphi)=\Hom(\psi).
\end{alignedat}    
\end{equation}
The following useful fact, taken from \cite{clari-uzca2023}, is easy to verify. 
$$
\begin{alignedat}{1}
\varphi H \psi & \iff \Hom(\varphi)\cap[X]^3=\Hom(\psi)\cap[X]^3.
\end{alignedat}    
$$
Our reconstruction problem is captured by  the following property (introduced in  \cite{clari-uzca2023}).
\begin{equation}
    \varphi \in\mathcal{R} \iff \forall\psi\, [ \varphi H \psi  \implies (\psi=\varphi \vee \psi=1-\varphi) ].
\end{equation} 
The colorings in $\mathcal{R}$ are those that can be reconstructed from its homogeneous sets in an (essentially) unique manner. 
The complement of $\mathcal{R}$ is denoted by $\neg \mathcal{R}$.

A \textit{critical pair for a coloring $\varphi$} is an edge $\{a,b\}\in [X]^{2}$ such that
\begin{equation}\label{eq:critpair}
\varphi\{a,x\} = 1-\varphi\{b,x\}, \text{ for all } x\in X\setminus\{a,b\}.    
\end{equation}

\begin{rem}
\label{rem-critical-pair}
Any  coloring with  a critical pair does not belong to $\mathcal{R}$ (\cite{clari-uzca2023}). Indeed, suppose that $\{a,b\}$ is critical for $\varphi$ and let $\psi$ be defined as $\psi\{x,y\}=\varphi\{x,y\}$ if $\{x,y\}\neq \{a,b\}$ and $\psi\{a,b\}=1-\varphi\{a,b\}$. It is easy to verify that  $\Hom(\varphi)=\Hom(\psi)$.
In figure \ref{fig:Pcrit} the pair $\{y,z\}$ is critical. Observe that the collection of its homogeneous sets does not change regardless of the color of $\{y,z\}$, thus this coloring is not in $\mathcal{R}$.  
\end{rem}

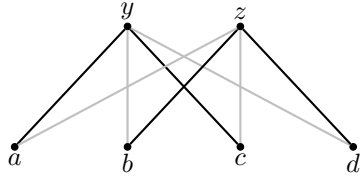
\begin{figure}[h!]
    \centering
    \begin{tikzpicture}[scale=1.5]
\node[above] (y) at (0,1) {\footnotesize$y$};
\node[above] (z) at (1,1) {\footnotesize$z$};

\node[below] (a) at (-1,0) {\footnotesize$a$};
\node[below] (b) at (0, 0) {\footnotesize$b$};
\node[below] (c) at (1,0) {\footnotesize $c$};
\node[below] (d) at (2,0) {\footnotesize $d$};

\bgroup
\tikzstyle{every node}=[circle, draw, fill=black,inner sep=0pt, minimum width=2.5pt]
\node[above] (y) at (0,1) {};
\node[above] (z) at (1,1) {};

\node[below] (a) at (-1,0) {};
\node[below] (b) at (0, 0) {};
\node[below] (c) at (1,0) {};
\node[below] (d) at (2,0) {};

\draw[thick] (c) -- (y) -- (a);
\draw[gray!50,thick] (d) -- (y) -- (b);
\draw[thick] (d) -- (z) -- (b);
\draw[gray!50,thick] (c) -- (z) -- (a);
\egroup

\end{tikzpicture}
    \caption{$\{y,z\}$ is a  critical pair.}
    \label{fig:Pcrit}
\end{figure}

A \textit{critical cycle} for a coloring $\varphi$ on $X$ is a set $Z = \{\{a,b\},\{b,c\},\{c,d\},\{d,a\}\}$ where $\{a,b,c,d\}\in [X]^4$ is such that 
\begin{equation}
\label{eq:critcycle1}
    \begin{alignedat}{2}
        \varphi\{a,c\}&=\varphi\{b,c\} = 1-\varphi\{a,b\},\\
        \varphi\{b,d\} &= \varphi\{c,d\} = 1-\varphi\{b,c\},\\
        \varphi\{c,a\} &= \varphi\{d,a\} = 1-\varphi\{c,d\}.
    \end{alignedat}
\end{equation}
and
\begin{equation}\label{eq:critcycle2}
   \varphi\{x,z\} \neq \varphi\{y,z\} \text{ for } \{x,y\} \in Z \text{ and } z\in X\setminus\{a,b,c,d\}. 
\end{equation}

For any $\{x,y\}\in [X]^2$,   define a set by 
\begin{equation}
\label{eq:criticos}
B_{\{x,y\}} = \{ z\in X\setminus\{x,y\}\,:\, \varphi\{x,z\}=\varphi\{y,z\} \}.
\end{equation}
Observe that if $\{x,y\}$ is a critical pair for $\varphi$, then $B_{\{x,y\}}$ is empty and if  
$\{x,y\}$ belongs to a critical cycle $Z = \{\{a,b\},\{b,c\},\{c,d\},\{d,a\}\}$, then $B_{\{x,y\}}$ has exactly one element, for instance, $B_{\{a,b\}}=\{c\}$.

\begin{rem}\label{rem:Ccrit}
    It is worth noting that from equation \eqref{eq:critcycle1} one can deduce
    \[\varphi\{d,b\}=\varphi\{a,b\}=1-\varphi\{d,a\}.\]

    Also, to avoid confusion, it is important to consider that the set $Z$ in itself might be a critical cycle by satisfying the equations in \eqref{eq:critcycle1} but also might be a critical cycle if it holds that
    \begin{equation}\label{eq:critcyclealt}
        \begin{alignedat}{2}
        \varphi\{b,d\}&=\varphi\{a,d\} = 1-\varphi\{a,b\},\\
        \varphi\{a,c\} &= \varphi\{d,c\} = 1-\varphi\{a,d\},\\
        \varphi\{d,b\} &= \varphi\{c,b\} = 1-\varphi\{d,c\}.
    \end{alignedat}
    \end{equation}
    So the same set $Z$ of four edges in $[X]^2$ might be a critical cycle in one of two ways. We will presuppose that equation \eqref{eq:critcycle1} is the condition that holds unless stated otherwise.
\end{rem}

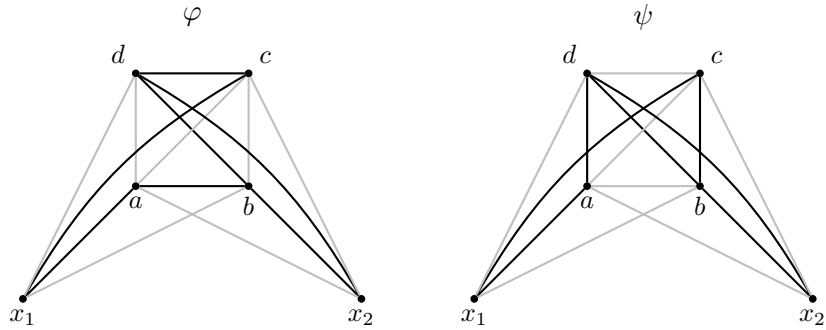
\begin{figure}[h]
\centering
\begin{tikzpicture}[scale=1.5]

\node (varphi) at (0.5, 1.5) {\small$\varphi$};
\node[above left] (d) at (0,1) {\footnotesize$d$};
\node[below] (a) at (0,0) {\footnotesize$a$};
\node[below] (b) at (1,0) {\footnotesize$b$};
\node[above right] (c) at (1,1) {\footnotesize$c$};

\node[below] (x1) at (-1,-1) {\footnotesize$x_1$};
\node[below] (x2) at (2,-1) {\footnotesize$x_2$};

\node (psi) at (4.5, 1.5) {\small$\psi$};
\node[above left] (d1) at (4,1) {\footnotesize$d$};
\node[below] (a1) at (4,0) {\footnotesize$a$};
\node[below] (b1) at (5,0) {\footnotesize$b$};
\node[above right] (c1) at (5,1) {\footnotesize$c$};

\node[below] (y1) at (3,-1) {\footnotesize$x_1$};
\node[below] (y2) at (6,-1) {\footnotesize$x_2$};

\bgroup
\tikzstyle{every node}=[circle, draw, fill=black,inner sep=0pt, minimum width=2.5pt]
\node (a) at (0,0) {};
\node (b) at (1,0) {};
\node (c) at (1,1) {};
\node (d) at (0,1) {};

\node (x1) at (-1,-1) {};
\node (x2) at (2,-1) {};

\node (a1) at (4,0) {};
\node (b1) at (5,0) {};
\node (c1) at (5,1) {};
\node (d1) at (4,1) {};

\node (y1) at (3,-1) {};
\node (y2) at (6,-1) {};

%%%%
\draw[thick] (a) -- (b) -- (d) -- (c);
\draw[gray!50, thick] (d) -- (a) -- (c) -- (b);

%%%%
\draw[thick] (c) to[out=210,in=60] (x1) -- (a);
\draw[thick, gray!50] (d) -- (x1) -- (b);
\draw[thick] (d) to[out=-30, in=120] (x2) -- (b);
\draw[thick, gray!50] (c) -- (x2) -- (a);

%%%%%
\draw[thick] (a1)-- (d1) -- (b1) -- (c1);
\draw[gray!50, thick] (b1) -- (a1) -- (c1) -- (d1);
%%%%
\draw[thick] (c1) to[out=210,in=60] (y1) -- (a1);
\draw[thick, gray!50] (d1) -- (y1) -- (b1);
\draw[thick] (d1) to[out=-30,in=120] (y2) -- (b1);
\draw[thick, gray!50] (c1) -- (y2) -- (a1);
\egroup

\end{tikzpicture}
\caption{$\varphi$ and $\psi$ are colorings on 6 vertices for which the set $Z=\{\{a,b\},\{b,c\},\{c,d\},\{d,a\}\}$ is a critical cycle. Notice that  $\varphi H\psi$, $D_1(\varphi+\psi)=Z$, $\varphi$ satisfies \eqref{eq:critcycle1} and $\psi$ satisfies \eqref{eq:critcyclealt}.}
    \label{fig:CriticalCycle}
\end{figure}

\begin{prop}\label{prop:critcycle}
Let $\varphi$ be a coloring on $X$. If $Z=\{\{a,b\},\{b,c\},\{c,d\}\{d,a\}\}$ is a critical cycle for $\varphi$, then $\varphi \in \neg\mathcal{R}$. 
\end{prop}

\begin{proof}
    Let $\psi$ be a coloring on $X$ defined by
    \[\psi\{x,y\} = \begin{cases}
         \varphi\{x,y\}, &\text{ if } \{x,y\}\notin Z\\
         1-\varphi\{x,y\}, &\text{ if } \{x,y\} \in Z.
    \end{cases}\]
    We want to prove that $\varphi H\psi$. Let $T\in [X]^3$. If $[T]^2 \cap Z = \emptyset$ then $\psi\restriction [T]^2 =\varphi\restriction [T]^2$, which means that $T \in \Hom(\varphi)$ if, and only if, $T\in \Hom(\psi)$.

    On the other hand, if $[T]^2\cap Z \neq \emptyset$, we have two cases.
\begin{itemize}
\item[(i)] If $T\not\subset \{a,b,c,d\}$, then $T=\{x,y,z\}$ where $\{x,y\}\in Z$ and $z\in X\setminus \{a,b,c,d\}$. By definition of $\psi$, we have $\psi\{x,z\}=\varphi\{x,z\}$ and $\psi\{y,z\}=\varphi\{y,z\}$. But equation \eqref{eq:critcycle2} says $\varphi\{x,z\}\neq \varphi\{y,z\}$. So $T$ is neither $\varphi-$homogenous nor $\psi-$homogenous. 

\item[(ii)] If $T\subset \{a,b,c,d\}$.  Then $T=\{x,y,z\}$ such that $\{x,y\},\{y,z\}\in Z$ and we also have $\varphi\{x,z\}=\varphi\{y,z\}=1-\varphi\{x,y\}$, according to equation \eqref{eq:critcycle1}. Then, by definition of $\psi$, we get
    \[ \psi\{x,y\} = \psi\{x,z\} = 1-\psi\{y,z\},\]
    which again means that $T$ is not $\varphi-$homogenous nor $\psi-$homogenous.    
\end{itemize}
\end{proof}

\begin{ex}[\cite{clari-uzca2023}]
\label{ex:PcritCcrit}
\begin{itemize}
\item[(i)] A coloring may have many critical pairs. In fact,  there exist colorings with infinite many critical pairs. For instance, consider a  partition of $\N$ into two infinite sets, say the even natural numbers $A_0 = \{2k\}_{k\in\N}$ and the odd natural numbers $A_1=\{2k+1\}_{k\in \N}$. Define $\varphi$ by
\[
\varphi\{x,y\}=1\iff \{x,y\}\subseteq A_i \text{ for some } i\in \{0,1\}.
\]
Let $x\in A_0$ and $y\in A_1$. Then $\{x,y\}$ is critical for $\varphi$. 

        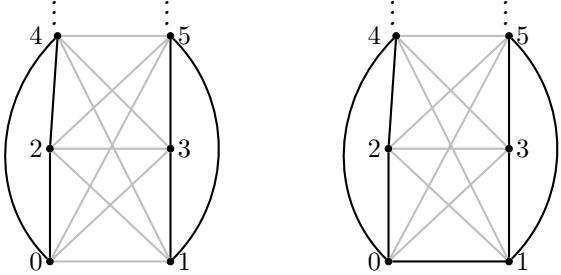
\begin{figure}[h]
            \centering
            \begin{tikzpicture}[scale=1.5]
\node[left] (0) at (0,0) {\footnotesize$0$};
\node[right] (1) at (1,0) {\footnotesize$1$};
\node[left] (2) at (0,1) {\footnotesize$2$};
\node[right] (3) at (1,1) {\footnotesize$3$};
\node[left] (4) at (0,2) {\footnotesize $4$};
\node[right] (5) at (1,2) {\footnotesize $5$};
\node[above] (l) at (0,2) {\small $\vdots$};
\node[above] (r) at (1,2) {\small $\vdots$};
\bgroup
\tikzstyle{every node}=[circle, draw, fill=black,inner sep=0pt, minimum width=2.5pt]

\node[left] (0) at (0,0) {};
\node[right] (1) at (1,0) {};
\node[left] (2) at (0,1) {};
\node[right] (3) at (1,1) {};
\node[right] (4) at (0,2) {};
\node[right] (5) at (1,2) {};

\draw[thick] (0) to (2) to (4) to[out=225, in=135] (0);

\draw[thick] (1) to (3) to (5) to[out=-45, in=45] (1);

\draw[gray!50,thick] (0)-- (1) -- (2) -- (3) -- (4) -- (5) -- (0) -- (3);

\draw[gray!50,thick] (1) -- (4);
\draw[gray!50,thick] (2) -- (5);
\egroup

\node[left] (0) at (3,0) {\footnotesize$0$};
\node[right] (1) at (4,0) {\footnotesize$1$};
\node[left] (2) at (3,1) {\footnotesize$2$};
\node[right] (3) at (4,1) {\footnotesize$3$};
\node[left] (4) at (3,2) {\footnotesize $4$};
\node[right] (5) at (4,2) {\footnotesize $5$};
\node[above] (l) at (3,2) {\small $\vdots$};
\node[above] (r) at (4,2) {\small $\vdots$};
\bgroup
\tikzstyle{every node}=[circle, draw, fill=black,inner sep=0pt, minimum width=2.5pt]

\node[left] (0) at (3,0) {};
\node[right] (1) at (4,0) {};
\node[left] (2) at (3,1) {};
\node[right] (3) at (4,1) {};
\node[right] (4) at (3,2) {};
\node[right] (5) at (4,2) {};

\draw[thick] (0) to (2) to (4) to[out=225, in=135] (0);

\draw[thick] (0) to (1) to (3) to (5) to[out=-45, in=45] (1);

\draw[gray!50,thick] (2) -- (3) -- (4) -- (5) -- (0) -- (3);

\draw[gray!50,thick] (5) -- (2) -- (1) -- (4);

\egroup
\end{tikzpicture}
            \caption{A coloring with infinite many critical pairs.}
            \label{fig:PartitionInducedColoring}
        \end{figure}

\item[(ii)] A coloring in $\neg\Rec$ that does not have critical pairs is described in Figure \ref{fig:no-critical-pair}. This coloring has, in fact, a critical cycle, given by $Z=\{\{0,1\},\{1,2\},\{2,3\},\{3,0\}\}$. Let $\varphi$ be the coloring represented in the figure below and $\psi$ the coloring obtained by changing the edges in $Z$, that is, $D_1(\varphi+\psi)=Z$, then $\varphi H\psi$. Note that the collection of homogenous sets for $\varphi$ is $\Hom(\varphi)=\{\{0,4,5,6\dots\},\{2,4,5,6,\dots\}\}$
\begin{figure}[h]
            \centering
       \begin{tikzpicture}[scale=1.5]
\node[left] (0) at (0,0) {\footnotesize$0$};
\node[right] (1) at (1,0) {\footnotesize$1$};
\node[right] (2) at (1,1) {\footnotesize$2$};
\node[left] (3) at (0,1) {\footnotesize$3$};
\node[right] (4) at (0.5,2) {\footnotesize $4$};
\node[above left] (5) at (0.5,3) {\footnotesize $5$};
\node[above] (m) at (0.5,3) {\small $\vdots$};

\bgroup
\tikzstyle{every node}=[circle, draw, fill=black,inner sep=0pt, minimum width=2.5pt]

\node (0) at (0,0) {};
\node (1) at (1,0) {};
\node (2) at (1,1) {};
\node (3) at (0,1) {};
\node (4) at (0.5,2) {};
\node (5) at (0.5,3) {};

\draw[thick] (0) to (3) to (1) to (2);
\draw[thick,gray!65] (3) to (2) to (0) to (1);

\draw[thick] (0) -- (4) -- (5) to[out=-70,in=90] (2) -- (4) -- (5) to (0);

\draw[gray!65,thick]  (1) -- (4) -- (3) to[out=90, in=240] (5) -- (1) ;

\egroup

\node[left] (0) at (3,0) {\footnotesize$0$};
\node[right] (1) at (4,0) {\footnotesize$1$};
\node[right] (2) at (4,1) {\footnotesize$2$};
\node[left] (3) at (3,1) {\footnotesize$3$};
\node[right] (4) at (3.5,2) {\footnotesize $4$};
\node[above left] (5) at (3.5,3) {\footnotesize $5$};
\node[above] (m) at (3.5,3) {\small $\vdots$};

\bgroup
\tikzstyle{every node}=[circle, draw, fill=black,inner sep=0pt, minimum width=2.5pt]

\node (0) at (3,0) {};
\node (1) at (4,0) {};
\node (2) at (4,1) {};
\node (3) at (3,1) {};
\node (4) at (3.5,2) {};
\node (5) at (3.5,3) {};

\draw[thick] (0) to (1) to (3) to (2);
\draw[thick] (0) -- (4) -- (5) to[out=-70,in=90] (2) -- (4) -- (5) to (0);

\draw[gray!65,thick] (1) -- (2) -- (0) to (3);
\draw[gray!65,thick]  (1) -- (4) -- (3) to[out=90, in=240] (5) -- (1) ;
\egroup
\end{tikzpicture}
            \caption{A coloring in $\neg \mathcal{R}$ without critical pair.}
            \label{fig:no-critical-pair}
        \end{figure}
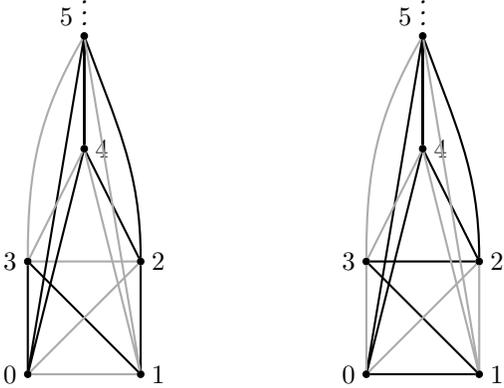
    \end{itemize}
\end{ex}

\section{Boolean sums}

The sum $\varphi+\psi$ of two colorings $\varphi, \psi\in 2^{[X]^2}$ is the usual sum module 2. They were called Boolean sums in \cite{Pouzetetal2011}, from where we take the terminology. Boolean sums $\varphi+\psi$ where  $\varphi H \psi$ will play a very important role in the whole paper. 
In this section we show a crucial result about them. 

\begin{prop}
\label{prop:Hom}
Let $\varphi$ and $\psi$ be colorings on $X$. Then  
\begin{itemize}
\item[(i)] $\Hom(\varphi)\cap\Hom(\psi)\subseteq\Hom(\varphi+\psi)$.

\item[(ii)]$\varphi H \psi \iff \Hom(\varphi)\cup \Hom(\psi) \subseteq \Hom(\varphi+\psi)$.
\end{itemize}

\end{prop}

\begin{proof}
(i) Let $\{x,y,z\} \in \Hom(\varphi)\cap \Hom(\psi)$. From the definitions of homogenous set and Boolean sum we have
\begin{align*}
        (\varphi+\psi)\{x,y\}=\varphi\{x,y\}+\psi\{x,y\}&=\varphi\{x,z\}+\psi\{x,z\}=(\varphi+\psi)\{x,z\}\\
        &=\varphi\{y,z\}+\psi\{y,z\}=(\varphi+\psi)\{y,z\}.
\end{align*}

(ii) $(\Longrightarrow)$ Suppose $\varphi H\psi$. Then, $  \Hom(\varphi)\cap \Hom(\psi)=\Hom(\varphi)\cup \Hom(\psi)$. The result follows from  part (i).

$(\Longleftarrow)$ Suppose $\Hom(\varphi)\cup\Hom(\psi) \subseteq \Hom(\varphi+\psi)$. Let $\beta=\varphi+\psi$, then $\varphi = \beta+\psi$ and $\psi=\beta+\varphi$. By  part (i) we have
\[
\Hom(\varphi)\cap\Hom(\beta)\subseteq \Hom(\varphi+\beta)=\Hom(\psi).
\]
Since $\Hom(\varphi)\subseteq \Hom(\beta)$, we have $\Hom(\varphi)\subseteq\Hom(\psi)$. Analogously, $\Hom(\psi)\subseteq\Hom(\varphi)$.
 Hence $\varphi H\psi$.
 \end{proof}

The following lemma is a key result on which several arguments in what follows depend. We were surprised to find out that it was already  known (see \cite[Lemma 2.2]{Pouzetetal2011}). We give an alternative proof.

\begin{lem}[\cite{Pouzetetal2011}]
\label{lem:Pouzet}
Let $\varphi,\psi$ be colorings on $X$. Then  $\varphi H \psi$ if, and only if, for every $\{x,y,z\} \in  [X]^3$
\begin{equation}
\label{eq:pouzet}
    (\varphi+\psi)\{x,y\} =(\varphi+\psi)\{x,z\}\neq (\varphi+\psi)\{y,z\} \implies \varphi\{x,y\} \neq \varphi\{x,z\}.
\end{equation}
\end{lem}

\begin{proof}
Suppose  $\varphi H\psi$. Towards  a contradiction suppose there are $x,y$ and $z$ such that
\begin{equation}\label{eq:notHom}
    (\varphi+\psi)\{x,y\}=(\varphi+\psi)\{x,z\} \neq (\varphi+\psi)\{y,z\} \quad \text{and} \quad \varphi\{x,y\}=\varphi\{x,z\}.
\end{equation}
Since $\{x,y,z\}\not\in\Hom(\varphi+\psi)$, from  proposition (\ref{prop:Hom}) we know $\{x,y,z\} \notin \Hom(\varphi)\cup\Hom(\psi)$. In particular,  $\{x,y,z\}\not\in \Hom(\varphi)$, but  $\varphi\{x,y\} = \varphi\{x,z\}$, thus
\begin{align*}
    \varphi\{x,y\}&= \varphi\{x,z\} = 1-\varphi\{y,z\}.
\end{align*}
By the definition of $\varphi+\psi$, we have that $\psi\{x,y\}=\psi\{x,z\}$ and by symmetry, 
\begin{align*}
    \psi\{x,y\}&=\psi\{x,z\}=1-\psi\{y,z\}.
\end{align*}
Thus, 
\begin{align*}
    (\varphi+\psi)\{x,y\} = 0 &\iff \varphi\{x,y\} = \psi\{x,y\}\\
    &\iff 1-\varphi\{y,z\} = 1-\psi\{y,z\} \\
    &\iff \varphi\{y,z\}=\psi\{y,z\}\\
    &\iff (\varphi+\psi)\{y,z\}=0
\end{align*}
which contradicts equation \eqref{eq:notHom}.

Conversely, suppose \eqref{eq:pouzet} holds. By proposition \ref{prop:Hom}, to get that $\varphi H\psi$, it suffices to show that 
\[ 
\hom(\varphi)\cup\hom(\psi) \subseteq \hom(\varphi+\psi).
\]
Let $\{x,y,z\}\notin \Hom(\varphi+\psi)$. By  permuting the vertices $x,y,z$ if necessary, we have
\[(\varphi+\psi)\{x,y\} = (\varphi+\psi)\{x,z\}\neq (\varphi+\psi)\{y,z\}.\]
By \eqref{eq:pouzet}, we have that $\varphi\{x,y\}\neq\varphi\{x,z\}$ and $\psi\{x,y\}\neq\psi\{x,z\}$. This, in turn, implies that $\{x,y,z\}\notin\Hom(\varphi)\cup\Hom(\psi)$. 
\end{proof}

\begin{figure}[h]
    \centering
    \begin{tikzpicture}[scale=1.0]
\node[above] (x) at (0,1) {\footnotesize$x$};
\node[below] (y) at (-1,0) {\footnotesize$y$};
\node[below] (z) at (1,0) {\footnotesize$z$};
\node[below] (phi) at (0,-0.5) {\footnotesize$\varphi+\psi$};

\node[above] (x1) at (4,1) {\footnotesize$x$};
\node[below] (y1) at (3,0) {\footnotesize$y$};
\node[below] (z1) at (5,0) {\footnotesize$z$};
\node[below] (phi1) at (4,-0.5) {\footnotesize$\varphi$};

\bgroup
\tikzstyle{every node}=[circle, draw, fill=black,inner sep=0pt, minimum width=2.5pt]
\node[above] (x) at (0,1) {};
\node[below] (y) at (-1,0) {};
\node[below] (z) at (1,0) {};
%\node[below] (phi) at (0,-0.5) {};

\node[above] (x1) at (4,1) {};
\node[below] (y1) at (3,0) {};
\node[below] (z1) at (5,0) {};
%\node[below] (phi1) at (4,-0.5) {};

\draw[thick] (y) -- (x) -- (z);
\draw[gray!50, thick] (y) -- (z);
\draw[thick] (x1) -- (y1);
\draw[gray!50,thick] (x1) -- (z1);
\egroup

\end{tikzpicture}
    \caption{Lemma \ref{lem:Pouzet} says that if $\varphi+\psi$ looks like the coloring on the left, or its complement, then $\varphi$ must look like the coloring on the right (or switching the colors of edges  $\{x,y\}$ and $\{x,z\}$, respectively).}
    \label{fig:lemaPouzet}
\end{figure}
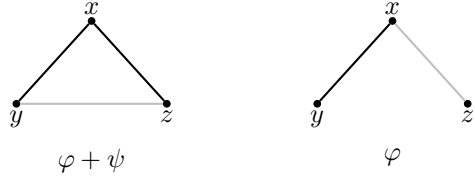

\section{Paths and cycles on components of boolean sums}
\label{sec:components}

The aim of this section is to show that,  for large enough $X$, if $\varphi H\psi$, the components of $D_1(\varphi+\psi)$ can  be only of two forms: either path or cycles of even length. This is a result from \cite{{Pouzet2013}}, however, our proof is different. On the other hand, we also show that if $\varphi H\psi$ and $C$ is a component of $D_1(\varphi+\psi)$, then $\varphi\restriction [C]^2$ is essentially a coloring given by a partition of $C$ (similar to  example \ref{ex:PcritCcrit}).

We need several results for the proof. They will be used also in the next section. 

\begin{lem}\label{lem:NoClaws}
Let $\varphi$ and $\psi$ be colorings  on $X$ such that $\varphi H\psi$. Then $\varphi+\psi$ does not contain claws. 
\end{lem}
\begin{proof}
Suppose $F=(\{x,y,z,w\}, [{x,y,z}]^2)$ is  a claw in $D_i(\varphi+\psi)$. Then by definition of a claw the following hold:
    \begin{align*}
        (\varphi+\psi)\{w,x\} &= (\varphi+\psi)\{w,y\} \neq  (\varphi+\psi)\{x,y\},\\
        (\varphi+\psi)\{w,x\} &= (\varphi+\psi)\{w,z\} \neq  (\varphi+\psi)\{x,z\},\\
        (\varphi+\psi)\{w,y\} &= (\varphi+\psi)\{w,z\} \neq  (\varphi+\psi)\{y,z\}.
    \end{align*}
By lemma \ref{lem:Pouzet}, we get $\varphi\{w,x\}\neq\varphi\{w,y\}$, $\varphi\{w,x\}\neq\varphi\{w,z\}$ and $\varphi\{w,y\}\neq \varphi\{w,z\}$. But the first two inequalities obtained imply $\varphi\{w,y\}=\varphi\{w,z\}$, which is a contradiction.
\end{proof}

The following proposition says that, for colorings on large enough sets,  all homogenous sets for a Boolean sum $\varphi+\psi$ of $H-$equivalent colorings have the same color. This fact will be needed for the proof of one of the main result in this paper. 

We recall that $R(k)$ is the {\em Ramsey number for homogenous sets of size $k$}, that is, $R(k)$ is the least integer such that any coloring on at least $R(k)$ vertices must have a homogenous set of size $k$. The case $k=7$ is important for the next result, we know that $ 205 \leq R(7) \leq 540$ (see \cite[table Ia, page 4]{Radz2017}).

\begin{lem}
\label{prop:Homi}
 Let $\varphi$ and $\psi$ be colorings on a set $X$  of size  at least $R(7)$   such that $\varphi H\psi$. Then there is $i\in \{0,1\}$ such that $\Hom(\varphi+\psi)=\Hom_i(\varphi+\psi)$.
\end{lem}

\begin{proof}
Let $\varphi H\psi$ and, to simplify the notation, let $D_i=D_i(\varphi+\psi)$ for $i=0,1$.  Suppose $\Hom_0(\varphi+\psi)\neq \emptyset, \Hom_1(\varphi+\psi)\neq \emptyset$. By hypothesis, there is an homogeneous set of size $7$ for $\varphi+\psi$. W.l.o.g.~ take $H_1\in \Hom_1(\varphi+\psi)\cap [X]^7$ and $H_0\in \Hom_0(\varphi+\psi)$. Then, by lemma \ref{lem:NoClaws}, for each $x\in H_0$,  we have
$$
\deg_{D_0\cap [H_1\cup\{x\}]^2}(x) \leq 2.
$$
Then, 
\[
\left|[H_0\cup H_1]^2 \cap D_0 \right|= \left|\bigcup_{x\in H_0} \{y\in H_1\,:\,\varphi+\psi\{x,y\}=0\}\right|\leq 2|H_0| \leq 6. 
\]
Since $H_1$ has size 7, there must be a $y\in H_1$ such that $\deg_{D_1\cap[H_0\cup\{x\}]^2}(y) = 3$, which implies that $(\varphi+\psi)\upharpoonright[H_0\cup\{y\}]^2$ contains a claw, contradicting lemma \ref{lem:NoClaws}.  
\end{proof}
The  result does not hold in general, but we do not know whether the bound $R(7)$ is an  optimal condition, it seems to be very loose one. Example \ref{fig:homsum} provides  a coloring $\varphi$ and one of its reconstruction $\psi$ such that  $\Hom(\varphi+\psi)\not=\Hom_i(\varphi+\psi)$ for $i=0,1$. Due to the previous result in several places we will assume w.l.o.g.~ that  $\Hom(\varphi+\psi)=\Hom_0(\varphi+\psi)$. 

\begin{figure}[H]
    \centering
    \begin{tikzpicture}[scale=1.5]

\node (varphi) at (0.5, -0.5) {\small$\varphi$};
\node[above] (a) at (0.5,2) {\footnotesize$a$};
\node[above left] (e) at (0,1) {\footnotesize$e$};
\node[below left] (d) at (0,0) {\footnotesize$d$};
\node[below right] (c) at (1,0) {\footnotesize$c$};
\node[above right] (b) at (1,1) {\footnotesize$b$};

\node (psi) at (2.5, -0.5) {\small$\psi$};
\node[above] (a1) at (2.5,2) {\footnotesize$a$};
\node[above left] (e1) at (2,1) {\footnotesize$e$};
\node[below left] (d1) at (2,0) {\footnotesize$d$};
\node[below right] (c1) at (3,0) {\footnotesize$c$};
\node[above right] (b1) at (3,1) {\footnotesize$b$};

\node (phi+psi) at (5, -0.5) {\small$\varphi+\psi$};
\node[above] (a1) at (5,2) {\footnotesize$a$};
\node[above left] (e1) at (4.5,1) {\footnotesize$e$};
\node[below left] (d1) at (4.5,0) {\footnotesize$d$};
\node[below right] (c1) at (5.5,0) {\footnotesize$c$};
\node[above right] (b1) at (5.5,1) {\footnotesize$b$};

\bgroup
\tikzstyle{every node}=[circle, draw, fill=black,inner sep=0pt, minimum width=2.5pt]
\node (a)  at (0.5,2) {};
\node (d) at (0,0) {};
\node (c) at (1,0) {};
\node (b) at (1,1) {};
\node (e) at (0,1) {};

\node (a1)  at (2.5,2) {};
\node (d1) at (2,0) {};
\node (c1) at (3,0) {};
\node (b1) at (3,1) {};
\node (e1) at (2,1) {};

\node (a2)  at (5,2) {};
\node (d2) at (4.5,0) {};
\node (c2) at (5.5,0) {};
\node (b2) at (5.5,1) {};
\node (e2) at (4.5,1) {};

%%%%
\draw[thick] (a) -- (b) -- (c) -- (a);
\draw[gray!65, thick] (a) -- (e) -- (d) -- (a);
\draw[thick] (c) -- (d);
\draw[thick] (b) -- (e);
\draw[thick,gray!65] (d) --(b);
\draw[thick,gray!65] (c) --(e);

\draw[thick] (a1) -- (b1) -- (c1) -- (a1);
\draw[thick] (a1) -- (e1) -- (d1) -- (a1);
\draw[thick,gray!65] (c1) -- (d1);
\draw[thick, gray!65] (b1) -- (e1);
\draw[thick,gray!65] (d1) --(b1);
\draw[thick,gray!65] (c1) --(e1);

\draw[thick, gray!65] (a2) -- (b2) -- (c2) -- (a2);
\draw[thick] (a2) -- (e2) -- (d2) -- (a2);
\draw[thick] (c2) -- (d2);
\draw[thick] (b2) -- (e2);
\draw[thick,gray!65] (d2) --(b2);
\draw[thick,gray!65] (c2) --(e2);

\egroup

\end{tikzpicture}
    \caption{A coloring for which the conclusion of lemma \ref{prop:Homi} does not hold. We have $\Hom(\varphi)=\Hom(\psi)=\Hom(\varphi+\psi)=\{\{a,b,c\},\{a,d,e\}\}$.}
    \label{fig:homsum}
\end{figure}
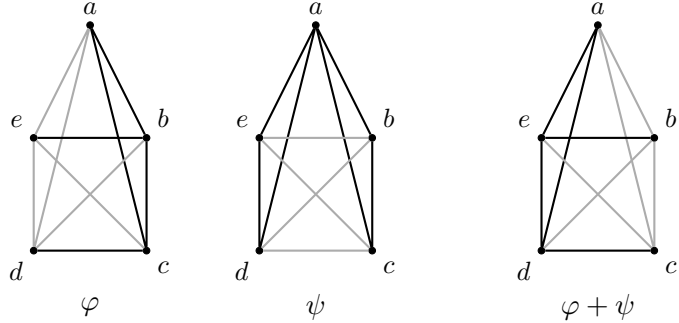

\begin{lem}
\label{teo:deg<=2}
Let $\varphi$ and $\psi$ be colorings on $X$ such that $\varphi H\psi$ and  $\Hom(\varphi+\psi)=\Hom_0(\varphi+\psi)$.  Then 
\[
\deg_{D_1(\varphi+\psi)}(x)\leq 2 \text{ for every $x\in X$}.
\]
\end{lem}

\begin{proof}
To simplify the notation, let $D_1 :=D_1(\varphi+\psi)$ and $D_0:= D_0(\varphi+\psi)$. Suppose that the conclusion does not hold.  Then there is $w\in X$ such that $\deg_{D_1}(w)>2$. Take any three edges $\{w,x\},\{w,y\},\{w,z\}\in D_1$.
If $(\varphi+\psi)\{x,y\}=1$, we would have $\{w,x,y\}\in \Hom_1(\varphi+\psi)$. Since there are no homogeneous sets of color $1$, it must be $\{x,y\},\{x,z\},\{y,z\}\in D_0$ and this is,  $\{x,y,z\}\in \Hom(\varphi+\psi)$. Letting $F=(\{x,y,z,w\},[\{x,y,z\}]^2)$ we have that $F$ is a claw, contradicting lemma \ref{lem:NoClaws}.
\end{proof}

The following three lemmas are all consequences of lemma \ref{lem:Pouzet}. The first of them was proved in \cite{Pouzet2013} (with slight differences) but was found independently by us along with lemma \ref{lem:Pouzet}. These lemmas are crucial for the rest of the article.
%used in the proofs of theorems \ref{teo:Pouzetetal} and \ref{teo:HomPartition}.

\begin{lem}[\cite{Pouzet2013}]\label{lem:path1}
Let $\varphi$ and $\psi$ be colorings on $X$ such that $\varphi H \psi$. Suppose there is a path $P=\{x_0,x_1,\dots,x_n\}\subseteq X$ of length $n\geq 2$ on $D_i(\varphi+\psi$). Then
     \begin{equation}\label{eq:path1}
        \varphi\{x_0,x_1\} \neq \varphi\{x_{n-1},x_n\} \iff n \text{ is even.}
    \end{equation}
\end{lem}

\begin{proof}
We assume w.l.o.g.~ that $i=1$. We proceed by induction on $n$. Suppose $n=2$. Since $P$ is a path on $D_1(\varphi+\psi)$, 
\[
(\varphi+\psi)\{x_0,x_1\} =(\varphi+\psi)\{x_1,x_2\} \neq (\varphi+\psi)\{x_0,x_2\}.
\]
Then by lemma \ref{lem:Pouzet} we get that  $\varphi\{x_0,x_1\}\neq\varphi\{x_1,x_2\}$.

Let $n\geq 2$ and $P=\{x_0,x_1,\dots, x_{n+1}\}$ be a path of length $n+1$ on $D_1(\varphi+\psi)$.  Suppose that equation \eqref{eq:path1} holds for $n$.  Consider the path $P_1=\{x_0,\dots,x_n\}$ and the path $P_2=\{x_{n-1}, x_n, x_{n+1}\}$ on $D_1(\varphi+\psi)$. Both paths are on $D_1(\varphi+\psi)$ and have length $n$ and length $2$, respectively. By the induction hypothesis (applied to $P_1$) $\varphi\{x_0,x_1\}=\varphi\{x_{n-1},x_n\}$ if and only if $n$ is odd. 

On the other hand, by the same argument as in the case $n=2$ but now applied to $P_2$, we get $\varphi\{x_{n-1},x_{n}\} \neq \varphi\{x_n,x_{n+1}\}$. We can then conclude that
\[
\begin{alignedat}{2}
\varphi\{x_0,x_1\} \neq \varphi\{x_n,x_{n+1}\} &\iff \varphi\{x_0,x_1\} = \varphi\{x_{n-1},x_n\} \\ &\iff n \text{ is odd }\\
        &\iff n+1\text{ is even.}
\end{alignedat} 
\]
\end{proof}
\begin{cor}
\label{cor:NoOddCycle} 
Let $\varphi$ and $\psi$ be colorings on a set $X$ such that $\varphi H \psi$. Any component of $X$ on $D_i(\varphi+\psi)$ that induces a cycle is of even length.
\end{cor}

\begin{proof}
We assume w.l.o.g.~ that $i=1$.
Let $C$ be a component of $X$ on $D_1(\varphi+\psi)$. 
From lemma \ref{lem:NoClaws} we already know that $C$ cannot induce a cycle of length 3, otherwise, there would be a claw on $D_1(\varphi+\psi)$.

Let $G$ be the graph induced on $D_1(\varphi+\psi)$ by $C$. Suppose $G$ is a cycle of length $2n+1$ with $n\geq 2$. Let $C=\{x_0,x_1,\dots, x_{2n}\}$. There are paths $P_1$, $P_2$ and $P_3$ on $D_1(\varphi+\psi)$ induced by the sets $\{x_0,x_1,x_2,x_3\}$, $\{x_2,x_3,\dots, x_{2n},x_0\}$ and $\{x_{2n},x_0,x_1\}$, respectively. Given that the length of $P_1$ is 3, the length of $P_2=2n-1$ and the length of $P_3=2$, by lemma \ref{lem:path1}, we have that 
\begin{equation}\label{eq:noOddcycle}
\begin{alignedat}{2}
   \varphi\{x_0,x_1\} &= \varphi\{x_{2},x_3\},\\
    \varphi\{x_2,x_3\} &= \varphi\{x_{2n},x_0\},\\
    \varphi\{x_{2n},x_0\} &\neq \varphi\{x_0,x_1\}.
\end{alignedat}
\end{equation}
Evidently, the third equation in \eqref{eq:noOddcycle} contradicts the previous two. We conclude that $C$ does not induce an odd cycle.
\end{proof}

\begin{lem}
\label{lem:path2}
Let $\varphi$ and $\psi$ be colorings on  $X$ such that $\varphi H\psi$. If $P=\{x_0,\dots,x_n\}$ is a path of length $n$ on $D_i(\varphi+\psi)$. Then 
\begin{equation}\label{eq:path2}
              \varphi\{x_0,x_2\} = \varphi\{x_0,x_n\} \iff n \text{ is even.}
\end{equation}
\end{lem}

\begin{proof} We assume w.l.o.g.~ that $i=1$.
For $n=2$ the equation holds trivially.

Suppose that equation \eqref{eq:path2} holds for $n\geq 2$.
Let $P=\{x_0,\dots,x_{n+1}\}$ be a path of length $n+1$ on $D_1(\varphi+\psi)$. Then, by the induction hypothesis
\[
\varphi\{x_0,x_2\} \neq \varphi\{x_0,x_n\} \iff n \text{ is odd.}
\]

Since $P$ is path on $D_1(\varphi+\psi)$, the only vertices that are $D_1-$ neighbors of $x_i$ are $x_{i-1}$ (if $i> 0$) and $x_{i+1}$ (if $i<n$). Then   
\[
(\varphi+\psi)\{x_0,x_n\} = (\varphi+\psi)\{x_0,x_{n+1}\} \neq (\varphi+\psi)\{x_n,x_{n+1}\}.
\]
By lemma \ref{lem:Pouzet}, we have that  $\varphi\{x_0,x_n\} \neq \varphi\{x_0,x_{n+1}\}$. So, 
    \[
    \begin{alignedat}{2}
         \varphi\{x_0,x_2\}=\varphi\{x_0,x_{n+1}\} &\iff \varphi\{x_0,x_2\} \neq \varphi\{x_0,x_n\} \\
         &\iff n \text{ is odd}\\
         &\iff n + 1 \text{ is even.}
    \end{alignedat}
    \]
This proves equation \eqref{eq:path2} for every $n\in \mathbb{N}$.
\end{proof}

\begin{lem}
\label{lem:evenpaths}
Let $\varphi$ and $\psi$ be colorings on $X$ such that $\varphi H\psi$. If $P=\{x_0,x_1,\dots,x_n\}$ is a path of length $n\geq 3$ on $D_c(\varphi+\psi)$, then 
\begin{equation}
\label{eq:evenpaths}
\varphi\{x_0,x_2\} = \varphi\{x_i,x_{i+2}\}=1-\varphi\{x_i,x_{i+3}\} = \varphi\{x_{i+1},x_{i+3}\}, \text{ for $0\leq i \leq n-3$.}
\end{equation}

\end{lem}

\begin{proof}
Assume w.l.o.g. we suppose that  $c=1$. The argument is by induction on the length of $P$. Suppose that equation \eqref{eq:evenpaths} holds for paths  of length $n-1$ and consider a path $P=\{x_0,\dots, x_n\}$ on $D_1(\varphi+\psi)$. 
Since $P\setminus\{x_n\}$ also induces a path on $D_1(\varphi+\psi)$,  we have that by the inductive hypothesis that 
\[
\varphi\{x_0,x_2\} =\varphi\{x_i,x_{i+2}\}=1-\varphi\{x_i,x_{i+3}\} = \varphi\{x_{i+1},x_{i+3}\} \text{ for every } i \in \{0,\dots,n-4\}.
\] 
On the other hand, let $P'=\{x_{n-3}, x_{n-2}, x_{n-1}, x_n\}$, then $P'$  is also a path on $D_1(\varphi+\psi)$. Thus  
\begin{align*}
    (\varphi+\psi)\{x_{n-3},x_n\}&=(\varphi+\psi)\{x_{n-2},x_n\} = 1-(\varphi+\psi)\{x_{n-3},x_{n-2}\}\\
    (\varphi+\psi)\{x_{n-3},x_n\}&=(\varphi+\psi)\{x_{n-3},x_{n-1}\} = 1-(\varphi+\psi)\{x_{n-1},x_n\}.
\end{align*} 
Then, by lemma \ref{lem:Pouzet}, we have 
\[
\varphi\{x_{n-3},x_{n-1}\} =1- \varphi\{x_{n-3},x_n\} = \varphi\{x_{n-2},x_n\}.
\]
We have established \eqref{eq:evenpaths}. 
\end{proof}

\begin{figure}[h]
    \centering
\begin{tikzpicture}[scale=1.5]

%%%
%%% Caminos
%%%
\node[below] (0) at (0,0) {\footnotesize$0$};
\node[above] (1) at (0,1) {\footnotesize$1$};
\node[below] (2) at (1,0) {\footnotesize$2$};
\node[above] (3) at (1,1) {\footnotesize$3$};
\node[below] (4) at (2,0) {\footnotesize $4$};
\node[above] (5) at (2,1) {\footnotesize $5$};

\node[below] (suma) at (1,-0.5) {\small $\varphi+\psi$};

\node[below] (a) at (3.5,1) {\footnotesize$0$};
\node[above] (b) at (3.5,2) {\footnotesize$1$};
\node[below] (c) at (4.5,1) {\footnotesize$2$};
\node[above] (d) at (4.5,2) {\footnotesize$3$};
\node[below] (e) at (5.5,1) {\footnotesize $4$};
\node[above] (f) at (5.5,2) {\footnotesize $5$};

\node[below] (i) at (3.5,-1) {\footnotesize$0$};
\node[above] (ii) at (3.5,0) {\footnotesize$1$};
\node[below] (iii) at (4.5,-1) {\footnotesize$2$};
\node[above] (iv) at (4.5,0) {\footnotesize$3$};
\node[below] (v) at (5.5,-1) {\footnotesize $4$};
\node[above] (vi) at (5.5,0) {\footnotesize $5$};
\node[right] at (6,1.5) {\small $\varphi$};
\node[right] at (6,-0.5) {\small $\psi$};

\bgroup
\tikzstyle{every node}=[circle, draw, fill=black,inner sep=0pt, minimum width=2.5pt]
%%%
%%% Caminos
%%%
\node (0) at (0,0) {};
\node (1) at (0,1) {};
\node (2) at (1,0) {};
\node (3) at (1,1) {};
\node (4) at (2,0) {};
\node (5) at (2,1) {};

\node (a) at (3.5,1) {};
\node (b) at (3.5,2) {};
\node (c) at (4.5,1) {};
\node (d) at (4.5,2) {};
\node (e) at (5.5,1) {};
\node (f) at (5.5,2) {};

\node (i) at (3.5,-1) {};
\node (ii) at (3.5,0) {};
\node (iii) at (4.5,-1) {};
\node (iv) at (4.5,0) {};
\node (v) at (5.5,-1) {};
\node (vi) at (5.5,0) {};

\draw[thick] (0) to (1) to (2) to (3) to (4) to (5);

\draw[thick] (a) to (b);
\draw[thick, gray!50] (b) to (c);
\draw[thick] (c) to (d);
\draw[thick, gray!50] (d) to (e);
\draw[thick] (e) to (f);

\draw[thick, gray!50] (i) to (ii);
\draw[thick] (ii) to (iii);
\draw[thick, gray!50] (iii) to (iv);
\draw[thick] (iv) to (v);
\draw[thick, gray!50] (v) to (vi);

\draw[thick] (a) -- (c) -- (e) to[out=-150, in=-30] (a);
\draw[thick] (b) -- (d) -- (f) to[out=150,in=30] (b) ;

\draw[thick] (i) -- (iii) -- (v) to[out=-150, in=-30] (i) ;
\draw[thick] (ii) -- (iv) -- (vi) to[out=150,in=30] (ii);

\draw[thick, gray!50] (b) to (e) -- (d) -- (a) -- (f) -- (c) ;

\draw[thick, gray!50] (v) -- (ii);
\draw[thick, gray!50] (i) -- (vi);
\draw[thick, gray!50] (i) -- (iv);
\draw[thick,gray!50]  (iii)--(vi);

\egroup

\end{tikzpicture}    
\caption{If $\varphi$ and $\psi$ are colorings such that $\varphi H\psi$ and there is a component $C$ that is a path (shown on the left) then lemmas \ref{lem:path1}, \ref{lem:path2} and \ref{lem:evenpaths}  determine $\varphi$ and $\psi$ on that component (shown on the right).}
    \label{fig:boolSum}
\end{figure}
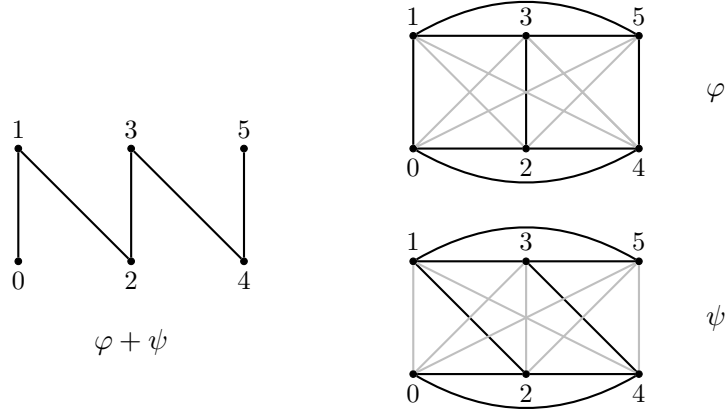

\begin{figure}[h]
    \centering
\begin{tikzpicture}[scale=1.5]

%%%
%%% Ciclos
%%%

\node[below] (0) at (0,-5) {\footnotesize$0$};
\node[above] (1) at (0,-4) {\footnotesize$1$};
\node[below] (2) at (1,-5) {\footnotesize$2$};
\node[above] (3) at (1,-4) {\footnotesize$3$};
\node[below] (4) at (2,-5) {\footnotesize $4$};
\node[above] (5) at (2,-4) {\footnotesize $5$};

\node[below] (suma) at (1,-5.5) {\small $\varphi+\psi$};

\node[below] (a) at (3.5,-4) {\footnotesize$0$};
\node[above] (b) at (3.5,-3) {\footnotesize$1$};
\node[below] (c) at (4.5,-4) {\footnotesize$2$};
\node[above] (d) at (4.5,-3) {\footnotesize$3$};
\node[below] (e) at (5.5,-4) {\footnotesize $4$};
\node[above] (f) at (5.5,-3) {\footnotesize $5$};

\node[below] (i) at (3.5,-6) {\footnotesize$0$};
\node[above] (ii) at (3.5,-5) {\footnotesize$1$};
\node[below] (iii) at (4.5,-6) {\footnotesize$2$};
\node[above] (iv) at (4.5,-5) {\footnotesize$3$};
\node[below] (v) at (5.5,-6) {\footnotesize $4$};
\node[above] (vi) at (5.5,-5) {\footnotesize $5$};
\node[right] at (6,-3.5) {\small $\varphi$};
\node[right] at (6,-5.5) {\small $\psi$};

\bgroup
\tikzstyle{every node}=[circle, draw, fill=black,inner sep=0pt, minimum width=2.5pt]

%%%
%%% Ciclos
%%%

\node (0) at (0,-5) {};
\node (1) at (0,-4) {};
\node (2) at (1,-5) {};
\node (3) at (1,-4) {};
\node (4) at (2,-5) {};
\node (5) at (2,-4) {};

\node (a) at (3.5,-4) {};
\node (b) at (3.5,-3) {};
\node (c) at (4.5,-4) {};
\node (d) at (4.5,-3) {};
\node (e) at (5.5,-4) {};
\node (f) at (5.5,-3) {};

\node (i) at (3.5,-6) {};
\node (ii) at (3.5,-5) {};
\node (iii) at (4.5,-6) {};
\node (iv) at (4.5,-5) {};
\node (v) at (5.5,-6) {};
\node (vi) at (5.5,-5) {};

\draw[thick] (0) to (1) to (2) to (3) to (4) to (5) -- (0);

\draw[thick] (a) to (b);
\draw[thick, gray!50] (b) to (c);
\draw[thick] (c) to (d);
\draw[thick, gray!50] (d) to (e);
\draw[thick] (e) to (f);

\draw[thick, gray!50] (i) to (ii);
\draw[thick] (ii) to (iii);
\draw[thick, gray!50] (iii) to (iv);
\draw[thick] (iv) to (v);
\draw[thick, gray!50] (v) to (vi);

\draw[thick] (a) -- (c) -- (e) to[out=-150, in=-30] (a);
\draw[thick] (b) -- (d) -- (f) to[out=150,in=30] (b) ;

\draw[thick] (i) -- (iii) -- (v) to[out=-150, in=-30] (i);
\draw[thick] (ii) -- (iv) -- (vi) to[out=150,in=30] (ii);
\draw[thick] (vi) -- (i);
\draw[thick, gray!50] (b) to (e) -- (d) -- (a) -- (f) -- (c) ;

\draw[thick, gray!50] (v) -- (ii);
%\draw[thick, gray!50] (i) -- (vi);
\draw[thick, gray!50] (i) -- (iv);
\draw[thick,gray!50]  (iii)--(vi);
\
\egroup

\end{tikzpicture}    \caption{If $\varphi$ and $\psi$ are such that $\varphi H\psi$ and there is a component $C$ that is cycle of even length (shown on the left) then lemmas \ref{lem:path1}, \ref{lem:path2} and \ref{lem:evenpaths} determine $\varphi$ and $\psi$ on that component (shown on the right).}
    \label{fig:boolSm2}
\end{figure}
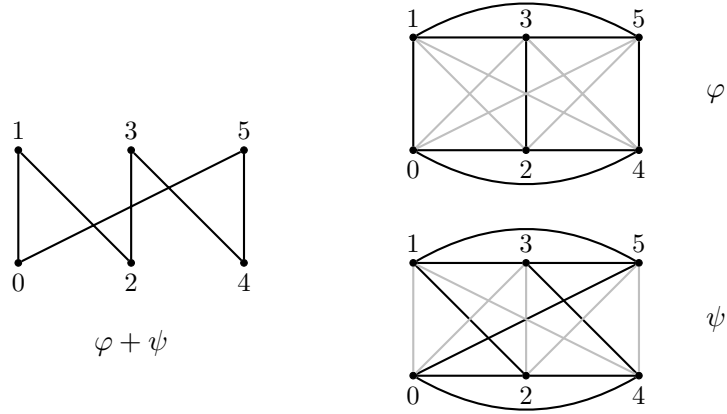

The following fact is probably known, we include a proof for the sake of completeness. 

\begin{prop}\label{prop:deg2}
Let $G$ be a graph on  $X$ such that $\deg_G(x)\leq 2$ for all $x$. Then every component of $G$ is  either a path or a cycle. 
\end{prop} 
\begin{proof} Let $C$ be a component of $G$.

\textbf{Case 1:}  Suppose there is $x_0\in C$ such that $\deg_{G}(x_0)=1$. Then there exists $x_1\in C$ such that $\{x_0,x_1\}\in E(G)$. Let  $C_1=\{x_0,x_1\}$. Suppose we have defined a path $C_k=\{x_0,x_1,\dots,x_k\}\subseteq C$ in $G$ for some $k\geq 1$. If $\deg_{G}(x_k)=1$, then $C_k$ is a component and thus $C=C_k$. Otherwise,  $\deg_{G}(x_k)=2$ and there is $x_{k+1}\in V(G)\setminus C$ with$\{x_k,x_{k+1}\}\in E(G)$. Let $C_{k+1}=C_k\cup\{x_{k+1}\}$. The process can either terminate after some finite number of steps $n$, in which case $C=C_n$, or go on indefinitely, in which case $C$ is an infinite path.

\bigskip 

\textbf{Case 2:} Suppose every $y\in C$ has degree $2$. Let $G'$ be the graph resulting from the removal from $G$ of some fixed element $y$. Since $\deg_{G}(y)=2$, there are $x$ and $z$ in $C$ such that $\{x,y\},\{y,z\} \in G$ and there are no other such elements in $X$. Then we have $\deg_{G'}(x)=\deg_{G'}(z)=1$, and $\deg_{G'}(w)=2$ for every $w\in C\setminus\{x,y,z\}$.
We consider two further subcases:

\textbf{Case 2.1:} $C$ is finite. By the first case we know that $G'$ is a finite path of length $|E(G)|-2$ from $y$ to $z$. The graph induced by $\{x,y,z\}$ on $E(G)$ is a path of length $2$. This means that $C$ induces a cycle.

\medskip 

\textbf{Case 2.2:} $C$ is infinite. Since $\deg_{G'}(x)=\deg_{G'}(z)=1$, then $x$ and $z$ belong to different components of $C\setminus\{y\}$ in $E(G')$. All other elements $w$ of $C\setminus\{x,y,z\}$ have degree 2 on $G'$, so each component must be an infinite path. There are only two components of $C\setminus\{y\}$ in $E(G')$ since $C$ was a connected component. Let $C_1$ be the component that induces the infinite path containing $x$, and $C_2$ the one containing $z$. On the other hand, when removing $y$ from $G$ we removed edges $\{x,y\}$ and $\{y,z\}$, so $G$ is the union of a path of length $2$ from $x$ to $z$ and two infinite paths with end vertices $x$ and $z$, respectively.
We conclude that $G$ is a two-way infinite path. 

\end{proof}

Now we show a  crucial result. The components of the Boolean sum  of two equivalent colorings of a large enough size (so that lemma \ref{prop:Homi} holds), in one of $D_1(\varphi+\psi)$ or $D_0(\varphi+\psi)$, are paths of arbitrary length or cycles of even length. Our proof is different than the one originally presented in \cite[Theorem 1.2]{Pouzetetal2011}.

\begin{teo}\cite{Pouzetetal2011}
\label{teo:Pouzetetal} 
Let $\varphi$ and $\psi$ be two colorings on $X$ such that $\Hom(\varphi+\psi)=\Hom_0(\varphi+\psi)$. If $\varphi H \psi$, then  the components of $X$ on $D_1(\varphi+\psi)$ are paths or cycles of even length.
\end{teo} 
\proof
From lemma \ref{teo:deg<=2} know that $\deg_{D_i(\varphi+\psi)}(x)\leq 2$ for every $x\in X$ and some $i\in \{0,1\}$. Then using proposition \ref{prop:deg2} on the graph $(X,D_i(\varphi+\psi))$ we deduce that every component of $(X,D_i(\varphi+\psi))$ is either a path or a cycle. Moreover,   corollary \ref{cor:NoOddCycle} says that if a component of $(X,D_i(\varphi+\psi))$ is a cycle, then it has even  length.
\endproof

We now turn our attention to the other main result of this section. We first establish a fact about pairs of $H$-equivalent colorings. 

\begin{lem}
\label{prop:path4}
Let $\varphi$ and $\psi$ be colorings on $X$ with $\varphi H \psi$ and  $\Hom(\varphi+\psi)=\Hom_0(\varphi+\psi)$. If $C$ is a component of $X$ on $D_1(\varphi+\psi)$ and $\{x,y,z\}\in [C]^3$, then there is a path $P=\{x_0,\dots,x_n\}$    on $D_1(\varphi+\psi)$ such that for some $1\leq j <n$, $\{x,y,z\}=\{x_0,x_j,x_n\}$.
\end{lem}

\begin{proof}
By theorem \ref{teo:Pouzetetal}, $C$ is either a path or a cycle of even length on $D_1(\varphi+\psi)$. 

\bigskip

\noindent \textbf{Case 1:} If $C$ is path on $D_1(\varphi+\psi)$ then there is a unique path between any two points. Then we find the largest path $P=\{x_0,\dots,x_n\}$ on $D_1(\varphi+\psi)$ between two elements of $\{x,y,z\}$. We permute $\{x,y,z\}$ such that $x_0=x$, $x_j=y$ for some $1\leq j < n$ and $x_n=z$.

\bigskip 

\noindent\textbf{Case 2:} If $C$ is a cycle of even length, then there are exactly two paths on $D_1(\varphi+\psi)$ between two non-adjacent elements of $C$. Permute $\{x,y,z\}$ so that $x$ and $z$ are non-adjacent, then $y$ is in one of the paths from $x$ to $z$. There is a path $P=\{x_0,\dots,x_n\}$ on $D_1(\varphi+\psi)$ such that $x_0=x$, $x_j=y$ for some $1\leq j < n$ and $x_n=z$.
\end{proof}

%\begin{lem}\label{lem:evenpaths}
%Let $\varphi$ and $\psi$ be colorings on $X$ such that $\varphi H\psi$ and $\Hom(\varphi+\psi)=\Hom_0(\varphi+\psi)$. If $C=\{x_0,x_1,\dots,x_n\}$ is a component of $D_1(\varphi+\psi)$ that is not a cycle of length 4, then 
%\begin{equation}\label{eq:evenpaths}
 %         \varphi\{x_0,x_2\} = \varphi\{x_i,x_{i+2}\}, \text{ for $0\leq i \leq n-2$.}
  %  \end{equation}
%\end{lem}

The following theorem gives some structural information regarding homogenous sets of the restriction $\varphi\restriction [C]^2$ where $C$ is a component of $D_1(\varphi+\psi)$. As seen in figures \ref{fig:boolSum} and \ref{fig:boolSm2}, there are only two $\subseteq-$chains of homogenous sets for $\varphi \restriction [C]^2$. The reader might have noticed there is a similarity with the coloring shown in figure \ref{fig:PartitionInducedColoring} in example \ref{ex:PcritCcrit} (i).

\begin{teo}
\label{teo:HomPartition}
Let $\varphi$ and $\psi$ be colorings on  $X$ with $\varphi H\psi$ and $\Hom(\varphi+\psi)=\Hom_0(\varphi+\psi)$. If $C$ is a component of $D_1(\varphi+\psi)$ with $|C|\geq 6$, then there are exactly two maximal homogeneous sets of the same color for $\varphi\upharpoonright[C]^2$ forming a partition of $C$.
\end{teo}

\begin{proof}
Let $C$ be a component of $X$ on $D_1(\varphi+\psi)$ with $|C|\geq 6$. The proof is based on the following fact. If $H\in [C]^{\geq 3}$, then 
\begin{equation}\label{eq:Hompath2}
\begin{alignedat}{2}
    H\in \Hom(\varphi) &\iff \forall \, x,z \in H \text{ there is a path of even length from $x$ to $z$.}\\
    &\iff \forall \, x,z \in H \text{ every path from $x$ to $z$ is of even length.}
\end{alignedat}     
\end{equation}

That the two statements on the right of \eqref{eq:Hompath2} are equivalent follows from theorem \ref{teo:Pouzetetal}. Indeed, given that $C$ is either a path or a cycle of even length on $D_1(\varphi+\psi)$, then either there is a unique path from $x$ to $z$ for $x,z\in C$ or there are exactly two paths from $x$ to $z$, and both have length of the same parity.

We need another auxiliary fact. Let us consider any triple $\{x,y,z\}\in [C]^3$. 
By lemma \ref{prop:path4} we know there is a path $P=\{x_0,\dots,x_n\}$ on $D_1(\varphi+\psi)$ such that $x_0=x$, $x_j=y$ for some $0< j < n$ and $x_n=z$. 
Define $P_1=\{x_0,\dots, x_j\}$ and $P_2=\{x_j,\dots, x_n\}$. The length of $P_1$ is $j$ and  the length of $P_2$ is $n-j$, as the length of $P$ is $n=j+n-j$. By lemma \ref{lem:evenpaths} we have 
\begin{equation}\label{eq:2by2}
   \varphi\{x_0,x_2\}=\varphi\{x_i,x_{i+2}\} \text{ for } 0\leq i \leq n-2. 
\end{equation}

Now we start with the  proof of \eqref{eq:Hompath2}. Let $H\in \Hom(\varphi\upharpoonright [C]^2)$ and  $\{x,y,z\}\in [H]^3$.  We are going to show that there is a path from $x$ to $y$ of even length, a path from $y$ to $z$ of even length, and a path from $x$ to $z$ of even length.

By lemma \ref{lem:path2}, we have that $\varphi\{x_0,x_2\} = \varphi\{x_0,x_n\}$ if and only if $n$ is even and that $\varphi\{x_0,x_2\}=\varphi\{x_0,x_j\}$ if and only if $j$ is even. Since identities $x_0=x$, $x_j=y$, $x_n=z$ hold and $\{x,y,z\}\in \Hom(\varphi)$,  $j$ and $n$ have same parity.

Let us see that $n$ must be even. Suppose not. Then $n$ and $j$ are odd and  $n-j$ is even. 
Equation \eqref{eq:2by2} and lemma \ref{lem:path2} applied to $P$, $P_1$ (odd length) and $P_2$ (even length) give
   \[ 
   \varphi\{x_0,x_j\} =\varphi\{x_0,x_n\} = 1-\varphi\{x_0,x_2\}=1-\varphi\{x_j,x_{j+2}\} = 1-\varphi\{x_j,x_n\},
   \]
which contradicts that $\{x,y,z\}$ is a homogenous set for $\varphi$. So $n$ and $j$  must be even and so $n-j$ is even too. 
   
Conversely, suppose that the right-hand side of \eqref{eq:Hompath2} holds for $H=\{x,y,z\}$. Again by lemma \ref{prop:path4}, we assume there is a path $P=\{x_0,\dots,x_n\}$ on $D_1(\varphi+\psi)$ such that $x_0=x$, $x_j=y$ for some $1\leq j<n$ and $x_n=z$. Then $n$ and $j$ are even.

Let $P_1=\{x_0,\dots,x_j\}$ and $P_2=\{x_j,\dots,x_n\}$ be the paths on $D_1(\varphi+\psi)$ from $x$ to $y$ and $y$ to $z$, respectively. $P_1$ and $P_2$ must have even length by hypothesis. Then, applying lemma \ref{lem:path2} we get
   \[ \varphi\{x_0,x_2\} =\varphi\{x_0,x_j\}=\varphi\{x_0,x_j\}=\varphi\{x_j,x_n\},\]
   which means that $\{x,y,z\}\in \Hom(\varphi)$.  So \eqref{eq:Hompath2} is proved.

Now we prove the theorem. Suppose $C=\{x_0,\dots,x_n\}$ is a component of $X$ on $D_1(\varphi+\psi)$ (either a path or a cycle of even length, by theorem \ref{teo:Pouzetetal}). Set 
\[ 
H_1=\{x_k \in C\,:\, k \text{ is even}\}\;\;\mbox{and}\;\; H_2=\{x_k \in C\,:\, k \text{ is odd}\}.
\]
By \eqref{eq:Hompath2}, both $H_1$ and $H_2$ are $\varphi$-homogeneous. If $x\in H_1$ and $y\in H_2$, then there is a path of odd length on $D_1(\varphi+\psi)$ from $x$ to $y$. Thus, by using again  \eqref{eq:Hompath2},  $H_1\cup \{y\}$ and $H_2\cup \{x\}$  are not $\varphi$-homogeneous. Since $C=H_1\cup H_2$ and $H_1$ and $H_2$ are disjoint, we conclude that $H_1$ and $H_2$ are maximally homogeneous sets for $\varphi$.

As  $\varphi\{x_0,x_2\}=\varphi\{x_1, x_3\}$, $H_1$ and $H_2$ are homogeneous of the same color. 
\end{proof}

\section{Minimal reconstruction of a coloring}
\label{sec:min-rec}

Let $\varphi$ and $\psi$ be colorings on $X$. To recall the  function $r$ defined in \cite{clari-uzca2023} we first set 
\[
D(\varphi, \psi)=\{ \{x,y\} \in X^{[2]}:  \psi\{x,y\}\neq \varphi\{x,y\}\}.
\]
It is important to realize that 
$$
D(\varphi, \psi)=D_1(\varphi+\psi).
$$
Define $r:(2^{X^{[2]}}\setminus\mathcal{R}) \to\mathbb{N}\cup\{\infty\}$  as 
\[
r(\varphi)= \min\{|D(\varphi, \psi)|: \varphi H\psi,  \,\psi\neq\varphi,\, \psi\neq1-\varphi\}. 
\]
We say that $\psi$ is a {\em  minimal (non trivial) reconstruction} of $\varphi$ if  $\varphi H\psi$ and $r(\varphi)=|D(\varphi,\psi)|$. 

It is immediate to see that if $\varphi$ has a critical pair, then $r(\varphi)=1$ (see remark \ref{rem-critical-pair}). Moreover, the converse also holds, if  $r(\varphi)=1$, then there is a critical pair for $\varphi$ (see Theorem 4.5 in \cite{clari-uzca2023}, we will proved it  below for the sake of completeness). 
It is also known that $r(\varphi)\neq 2$ for all $\varphi\not\in \mathcal{R}$ (see Lemma 4.4 in \cite{clari-uzca2023}).   We are going to show that  there are only two possibilities: Either $r(\varphi)=1$ or $r(\varphi)=4$. Our strategy consists of showing that  there is a critical pair for any coloring $\varphi$ such the components of the Boolean sum $\varphi+\psi$ does not contain a cycle of length 4  whenever $\varphi H \psi$.  This means that $r(\varphi)=1$ for these colorings.

In example \ref{ex:PcritCcrit} two colorings and some of their non-trivial reconstructions were defined and illustrated. The coloring $\varphi$ in item (i) of \ref{ex:PcritCcrit} exemplifies the case $r(\varphi)=1$. The coloring $\varphi$ in (ii) is shown to have a reconstruction $\psi$ such that $|D_1(\varphi+\psi)|=4$ (it has a critical cycle) and it is in fact a minimal reconstruction by the results we show in this section.

The next results shows the relevance of the components of $D_1(\varphi+\psi)$, with $\psi$ a non-trivial reconstruction of $\varphi$, to determine  the value of $r(\varphi)$

\begin{lem}
\label{recon-comp}
Let $\varphi$ and $\psi$ be colorings on $X$. Suppose  $\varphi H\psi$,  $\psi\neq\varphi$,  $\psi\neq1-\varphi$
and $r(\varphi)= |D(\varphi, \psi)|$.
Let $C$ be a component of $D_i(\varphi+\psi)$ for some $i\in\{0,1\}$. Suppose   $\psi\restriction [C]^2$ is a non-trivial reconstruction of $\varphi\restriction [C]^2$. Let $\psi'$ be the coloring on $X$ defined by 
\[
\psi'\{x,y\}= \begin{cases}
\psi\{x,y\} &\mbox{if $\{x,y\}\subseteq C,$}\\
\varphi\{x,y\}&  \mbox{if $\{x,y\}\not\subseteq C$}.
\end{cases}
\]
Then $\varphi H\psi'$ and 
$$
r(\varphi)\leq |D(\varphi\restriction [C]^2, \psi'\restriction [C]^2)|.
$$
In particular,  if $\psi$ is a minimal reconstruction of $\varphi$, then $D_1(\varphi+\psi)$ is connected.
\end{lem}

\begin{proof} To see that $\varphi H\psi'$, let $T=\{x,y,z\}\in [X]^3$ and consider two cases. 

\bigskip 

\noindent\textbf{Case 1:} Suppose $T\subseteq C$.  Then  $T\in \Hom(\varphi)$ iff $T\in \Hom(\psi')$,  as $\psi$ and $\psi'$ are equal on $[C]^2$. 

\bigskip 

\noindent\textbf{Case 2:} Suppose $T\not\subseteq C$. If $|T\cap C|\leq 1$, then $\varphi$ and $\psi'$ are equal on $[T]^2$. Finally, suppose w.l.o.g. that $T\cap C= \{x,y\}$. As $C$ is a component of $\varphi+\psi$, by lemma \ref{lem:Pouzet}, we conclude that 
$$
\psi'\{x,z\}=\varphi\{x,z\}\neq \varphi\{y,z\}=\psi'\{y,z\}.
$$
Thus $T\not\in \Hom(\varphi)\cup \Hom(\psi')$. 

\bigskip 

We have shown that $\varphi H\psi'$. 
Since $\psi\restriction [C]^2$ is a nontrivial reconstruction of $\varphi\restriction [C]^2$, we have that $\psi'\neq \varphi$ and $\psi'\neq 1-\varphi$.  Clearly $r(\varphi)\leq |D(\varphi\restriction [C]^2, \psi'\restriction [C]^2)|$.

Finally,  suppose that $\psi$ is a minimal reconstruction of $\varphi$, and that $D_1(\varphi+\psi)$ has more than one component. Then take an arbitrary component $C$. By the preceding argument, $r(\varphi)\leq D_1(\varphi \restriction [C]^2+\varphi\restriction [C]^2) < D_1(\varphi+\psi)$, contradicting the minimality of $\psi$.
\end{proof}

\begin{lem}
\label{lem:notinC}
Let $\varphi$ and $\psi$ be colorings on a set $X$ with $\varphi H\psi$ and $C$ be  a component of $D_1(\varphi+\psi)$. Then 
\begin{itemize}
\item[(i)] $\varphi\{x,y \}\neq\varphi\{x,z\}$ for every $x \notin C$ and every $\{y,z\}\in [C]^2 \cap D_1(\varphi+\psi)$.

\item[(ii)] If $\{y,z\} \in D_1(\varphi+\psi)$ is a critical pair for $\varphi\upharpoonright [C]^2$, then $\{y,z\}$ is a critical pair for $\varphi$.

\item[(iii)] If $\{y,z\}\in [X]^2$ is a component of $D_1(\varphi+\psi)$, then $\{y,z\}$ is a critical pair for $\varphi$. 

\item[(iv)]If $r(\varphi)=1$, then $\varphi$ has a critical pair.
\end{itemize}
\end{lem}
\proof
(i) If $x\notin C$, then $(\varphi+\psi)\{x,y\} =(\varphi+\psi)\{x,z\} =0$ and $(\varphi+\psi)\{y,z\}=1$. Then, by lemma \ref{lem:Pouzet}, $\varphi\{x,y\}\neq \varphi\{x,z\}$.

(ii) and (iii)  follow immediately from (i).

(iv) Suppose  $r(\varphi)=1$ and let $\psi$ be such that $|D_1(\varphi,\psi)|=|D(\varphi, \psi)|=1$. Let $y,z\in X$ be such that $\{\{y,z\}\}=D_1(\varphi+\psi)$. Thus $\{y,z\}$ is a component of $D_1(\varphi+\psi)$ and the result follows from (iii).
\endproof

The following lemma says that, for any given $\{x,y\}\in [C]^2\cap D_1(\varphi+\psi)$,  in order to determine whether  $\{x,y\}$ is a critical pair or not, it is enough to check the edges that are $D_1$- adjacent to $\{x,y\}$ .

\begin{lem}\label{lem:lemcrit}
    Let $\varphi$ and $\psi$ be colorings on $X$ with $\Hom(\varphi+\psi)=\Hom_0(\varphi+\psi)$. Suppose that $C$ is a  component of $X$ on $D_1(\varphi+\psi)$ such that $C$ is not a cycle of length 4. Then $\varphi$ has a critical pair.
\end{lem}
\begin{proof}
If $|C|=2$ then $C$ is just one edge, and by lemma \ref{lem:notinC}, (iii), there is a critical pair for $\varphi$. So suppose $|C|\geq 3$. Let $P=\{x,y,z\}\subseteq C$ be a path of length 2 on $D_1(\varphi+\psi)$, and suppose $\varphi\{x,y\}\neq \varphi\{x,z\}$.  We will show that $\{y,z\}$ is a critical pair for $\varphi$. Note that in the opposite case, $\varphi\{x,y\}=\varphi\{x,z\}$, we would have $\varphi\{y,z\}\neq \varphi\{y,z\}$ (by  lemma \ref{lem:Pouzet}), and the argument would be completely analogous to show that  $\{x,y\}$ is a critical pair for $\varphi$.

Suppose  $|C|=3$, i.e. $C=P$. Notice that  $\{y,z\}\in D_1(\varphi+\psi)$ and it is a critical pair for $\varphi\restriction [C]^2$. Thus,   by lemma \ref{lem:notinC} (ii), $\{y,z\}$ is a critical pair for $\varphi$. 

Suppose $|C|\geq 4$, we will again prove that $\{y,z\}$ is a critical pair for $\varphi \restriction [C]^2$. Let $w\in C\setminus P$. Note that $(\varphi+\psi)\{w,y\}=0$, since $\deg_{D_1(\varphi+\psi)}(y)<3$. We consider two cases. 

\begin{itemize}
\item[(i)] If $(\varphi+\psi)\{w,z\}=0$,   we use lemma \ref{lem:Pouzet} to get $\varphi\{y,w\}\neq \varphi\{w,z\}$.

\item[(ii)] If $(\varphi+\psi)\{w,z\}=1$,  then $P'=P\cup\{w\}$ is a path of length 3 on $D_1(\varphi+\psi)$ (since $C$ must be either a cycle of even length greater than 4 or a path by theorem \ref{teo:Pouzetetal}). By equation \eqref{eq:evenpaths} of lemma \ref{lem:evenpaths}, applied to $P'$, we have $\varphi\{x,z\}=\varphi\{y,w\}$. By lemma \ref{lem:path1} (also applied to $P'$),  $\varphi\{x,y\} = \varphi\{z,w\}$. Since we assumed that $\varphi\{x,y\}\neq \varphi\{x,z\}$, we conclude that $\varphi\{y,w\}=\varphi\{x,z\}=1-\varphi\{x,y\} =\varphi\{z,w\}$.  

\end{itemize}

We just proved that for every $w\in C\setminus\{y,z\}$, $\varphi\{y,w\}\neq \varphi\{z,w\}$. Applying lemma \ref{lem:notinC}, (ii), we conclude $\{y,z\}$ is a critical pair for $\varphi$.

\end{proof}

\begin{teo}
\label{teo:SumasMinimales}
Let $X$ be a set of size at least $R(7)$ and  $\varphi$ be a coloring on $X$ with $\varphi\in \neg \mathcal{R}$. Then $r(\varphi)=1$ or $r(\varphi)=4$. 
\end{teo}
\begin{proof}
    Let $\psi$ be a minimal non-trivial reconstruction of $\varphi$. By lemma \ref{recon-comp} there is a unique component $C=\{x_0,x_1,\dots,x_n\}$ on $D_1(\varphi+\psi)$. 

    Suppose $r(\varphi)\neq 1$. Then, by lemma \ref{lem:lemcrit}, we have that $C$ must be a cycle of length $4$.
 Thus $r(\varphi)=|D_1(\varphi+\psi)| = 4$.
\end{proof}

\begin{teo}
\label{teo:SumasMinimales2}
Let $\varphi$ be a coloring on  a set $X$ of size at least $R(7)$. If  $\varphi\in \neg \mathcal{R}$ and $r(\varphi)=4$, then there is $\psi$ such that $\varphi H\psi$ and  $D_1(\varphi+\psi)$ is a critical cycle for $\varphi$.
\end{teo}
\proof
Let $\psi$ be a minimal non-trivial reconstruction of $\varphi$. As in the proof of theorem \ref{teo:SumasMinimales}, the unique component  $C$ of $X$ on $D_1(\varphi+\psi)$) is a cycle of length 4. We will show it is critical for $\varphi$. 

Let $C=\{a,b,c,d\}$ and $D_1(\varphi+\psi)=\{\{a,b\},\{b,c\},\{c,d\},\{d,a\}\}$. By lemma \ref{lem:path1}, we get
\begin{equation}\label{eq:cyclecrit}
    \varphi\{a,b\} = 1-\varphi\{b,c\} = \varphi\{c,d\} = 1-\varphi\{d,a\}.
\end{equation}

We need to check two cases:

\medskip 

\noindent\textbf{Case 1:}  $\varphi\{a,c\}=\varphi\{b,d\}$. We will show that this case does not occur. In fact, 
suppose first that $\varphi\{a,b\}=\varphi\{a,c\}$. Then $\varphi\{a,c\}=1-\varphi\{b,c\}$ and $\varphi\{a,d\}=1-\varphi\{b,d\}$. By part (ii) of lemma \ref{lem:notinC}, we got that $\{a,b\}$ is a critical pair for $\varphi$, i.e. $r(\varphi)=1$, which is impossible.  If $\varphi\{a,b\}=1-\varphi\{a,c\}$, a completely analogous argument shows that $\{b,c\}$ is a critical pair for $\varphi$.

\medskip

\noindent \textbf{Case 2:} $\varphi\{a,c\} \neq\varphi\{b,d\}$. Suppose $\varphi\{a,b\}=\varphi\{b,d\}$, then by \eqref{eq:cyclecrit} we have
\[ 
\begin{alignedat}{3}
    \varphi\{a,c\} &= \varphi\{b,c\} &=1-\varphi\{a,b\}, \\
    \varphi\{b,d\} &= \varphi\{c,d\} &= 1-\varphi\{b,c\},\\
    \varphi\{c,a\} &=\varphi\{d,a\} &= 1-\varphi\{c,d\},
\end{alignedat}
\]
which is precisely equation \eqref{eq:critcycle1} in the definition of a critical cycle. 
The subcase when  $\varphi\{b,c\}=\varphi\{b,d\}$ is completely analogous, it will give equation \eqref{eq:critcyclealt} in remark \ref{rem:Ccrit}, and we again conclude that $D_1(\varphi+\psi)$ is a critical cycle for $\varphi$.
\endproof

\begin{ex}[Pairs of colorings with several components of $D_1(\varphi+\psi)$]

It is worth noting that it is possible to construct examples of colorings $\varphi$ and $\psi$ for which there is more than one component of $D_1(\varphi+\psi)$.

\begin{itemize}
    \item[(i)] The first of the examples was already given in figure \ref{fig:CriticalCycle}. For the colorings shown there $Z=\{a,b,c,d\}$ is a critical cycle but, on the other hand, $\{x_1,x_2\}$ is a critical pair.

    \item[(ii)] This next example (see figure \ref{fig:two-cycles}) consists of a coloring $\varphi$ on 8 vertices such that there are two critical cycles for $\varphi$, $Z=\{a,b,c,d\}$ and $Z_1=\{a_1,b_1,c_1,d_1\}$.
\end{itemize}
    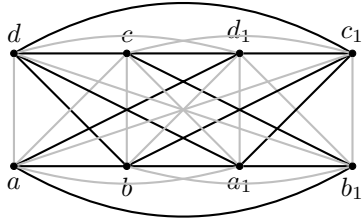
\begin{figure}[h]
     \centering
     \begin{tikzpicture}[scale=1.5]

\node[above] (d) at (0,1) {\footnotesize$d$};
\node[below] (a) at (0,0) {\footnotesize$a$};
\node[below] (b) at (1,0) {\footnotesize$b$};
\node[above] (c) at (1,1) {\footnotesize$c$};

\node[above] (d1) at (2,1) {\footnotesize$d_1$};
\node[below] (a1) at (2,0) {\footnotesize$a_1$};
\node[below] (b1) at (3,0) {\footnotesize$b_1$};
\node[above] (c1) at (3,1) {\footnotesize$c_1$};

\bgroup
\tikzstyle{every node}=[circle, draw, fill=black,inner sep=0pt, minimum width=2.5pt]
\node (a) at (0,0) {};
\node (b) at (1,0) {};
\node (c) at (1,1) {};
\node (d) at (0,1) {};

\node (a1) at (2,0) {};
\node (b1) at (3,0) {};
\node (c1) at (3,1) {};
\node (d1) at (2,1) {};

%%%%
\draw[thick] (a) -- (b) -- (d) -- (c);
\draw[gray!50, thick] (d) -- (a) -- (c) -- (b);

%%%%%
\draw[thick, gray!50] (a1)-- (d1) -- (b1) -- (c1);
\draw[thick] (b1) -- (a1) -- (c1) -- (d1);

%%%
\draw[thick] (c) -- (d1) -- (a) to[out=-30,in=210] (b1) -- (c);
\draw[thick] (b) -- (a1) -- (d) to[out=30,in=150] (c1) -- (b);

%%%
\draw[thick, gray!50] (a) -- (c1) to[out=165,in=15] (c) -- (a1) to[out=195,in=-15] (a);
\draw[thick, gray!50] (b) -- (d1) to[out=165,in=15] (d) -- (b1) to[out=195,in=-15] (b);

\egroup

\end{tikzpicture}
     \caption{A coloring with two critical cycles}
     \label{fig:two-cycles}
 \end{figure}
\end{ex}
 
\section{Strongly reconstructible colorings}

The following result from  \cite[Proposition 3.1]{clari-uzca2023} motivates the contents of this section.

\begin{prop}
\label{con-finita}
Let $\varphi$ be a coloring on $X$. If for every $F\in [X]^4$ there is $Y\subseteq X$ such that $F \subseteq Y$ and $\varphi \restriction [Y]^2\in \mathcal{R}$, then $\varphi \in \mathcal{R}$.
\end{prop}

We recall a sufficient condition for being in $\mathcal{R}$ introduced in \cite{clari-uzca2023} which is related to the previous result. A coloring $\varphi$ on $X$ has property $E_i$ if for every finite set $F\subseteq X$ there is $z\in X\setminus F$ such that $\varphi\{x,z\}=i$ for all $x\in F$.  Every coloring with property $E_i$ belongs to $\mathcal{R}$ (see \cite{clari-uzca2023}), precisely, because it satisfies the hypothesis in proposition \ref{con-finita}.  
This inspired the following question.

\begin{question}
\label{que:SR}\cite{clari-uzca2023}
Let $\varphi$ be a coloring on $\N$ with the property $\mathcal{R}$ 
and $F\subseteq \N$ be a finite set. Is there a finite set $G\supseteq F$ such that $\varphi\restriction [G]^2\in \mathcal{R}$?
\end{question}

A coloring  $\varphi$ on $\N$ has the property $\mathcal{SR}$ ({\em strongly reconstructible}) if for every $F\in [\N]^4$ there is finite $G \supseteq F$ such that $\varphi\restriction [G]^2\in \mathcal{R}$. Proposition \ref{con-finita} says that $\mathcal{SR}$ implies $\mathcal{R}$. 
We will provide an example of a coloring that answer question \ref{que:SR} negatively which shows that $\mathcal{SR}$ is strictly stronger than $\mathcal{R}$. 

The next result follows immediately from proposition \ref{con-finita}.

\begin{prop}\label{eq:EiSR}
Let $\varphi$ be a coloring on $X$. If $\varphi$ has the property $E_i$ for some $i\in \{0,1\}$ then $\varphi$ has the property $\mathcal{SR}$.
\end{prop}

Our next result provides a characterization of $\neg\mathcal{R}$.

\begin{teo}
\label{prop:RnotSRcol}
Let $\varphi$ be a coloring on $X$ where $|X|\geq R(7)$. The following statements are equivalent:
\begin{itemize}
    \item[(a)] $\varphi \notin \mathcal{R}$.
    \item[(b)] there is $F\in [X]^4$  such that $\varphi\restriction [F]^2$ has a non-trivial reconstruction $\psi_F$  and for every finite subset $G$ of $X$ with $G\supseteq F$ there exists a non-trivial reconstruction $\psi_G$ for $\varphi\restriction [G]^2$ such that 
$$
D_1(\varphi\restriction [G]^2 +\psi_G)=D_1(\varphi\restriction [F]^2+\psi_F)\neq \emptyset.
$$
   \item[(c)] there is $F\in [X]^4$  such that $\varphi\restriction [F]^2$ has a non-trivial reconstruction $\psi_F$  and for every $G\in[X]^7$  with $G\supseteq F$ there exists a non-trivial reconstruction $\psi_G$ for $\varphi\restriction [G]^2$ such that 
$$
D_1(\varphi\restriction [G]^2 +\psi_G)=D_1(\varphi\restriction [F]^2+\psi_F)\neq \emptyset.
$$
\end{itemize}
\end{teo}

\proof
    \begin{description}
        \item[(a) $\implies$ (b)] Let $\psi$ be a minimal non-trivial reconstruction of $\varphi$. Then $D_1(\varphi+\psi)$ is either a critical pair or critical cycle as shown by theorems \ref{teo:SumasMinimales} and \ref{teo:SumasMinimales2}. Let $C$ be the unique component of $D_1(\varphi+\psi)$, thus $|C|\leq 4$. Let $F\in [X]^4$ be such that $C\subseteq F$. Then for every finite $G\supseteq F$ put $\psi_G = \psi\restriction [G]^2$. Clearly, the conclusion holds.

        \item[(b) $\implies$ (c)] This is immediate.

        \item[(c) $\implies$ (a)]  Let  $F\in [X]^4$ and $\psi_F$ be a non-trivial reconstruction of $\varphi\restriction [F]^2$ such that for every finite set  $G$ containing  $F$ there exists a non-trivial reconstruction $\psi_G$ of $\varphi\restriction [G]^2$ such that  
 $$
 D_1(\varphi\restriction [G]^2 +\psi_G)=D_1(\varphi\restriction [F]^2 + \psi_F)\neq \emptyset.
 $$
 The last inequality is because these reconstructions are non-trivial. 
 
 Consider the coloring $\psi:[X]^2 \to \{0,1\}$ defined by
 \[ 
 \psi\{x,y\} = \begin{cases}
     1-\varphi\{x,y\}, &\text{ if } \{x,y\} \in D_1(\varphi\restriction [F]^2+\psi_F)\\
     \varphi\{x,y\} &\text{ if } \{x,y\} \notin D_1(\varphi\restriction [F]^2+\psi_F).
 \end{cases}
 \]
We will show that $\psi$ is a non trivial reconstruction of $\varphi$. 
We need some  facts for the proof.

\begin{itemize}
\item[(i)] For any $G\in [X]^7$ with  $G\supseteq F$, we have 
$$
D_1(\varphi+\psi)=D_1(\varphi\restriction [G]^2+\psi_G)\neq \emptyset.
$$
This holds by the definition of $\psi$ and the hypothesis that $D_1(\varphi\restriction [G]^2 +\psi_G)=D_1(\varphi\restriction [F]^2+\psi_F)\neq \emptyset$. 

\item[(ii)] We have that $\psi\restriction [G]^2 = \psi_G$, otherwise there would be an edge $\{x,y\}\in [G]^2$ such that $\{x,y\} \in D_1(\varphi+\psi) \Delta D_1(\varphi\restriction [G]^2 +\psi_G)$, which is impossible.

\item[(iii)] We claim that  $\varphi H\psi$. Indeed, let $\{x,y,z\}\in \Hom(\varphi)$ and   $G=F\cup\{x,y,z\}$, so $\{x,y,z\}\in \Hom(\varphi\restriction [G]^2)$. If $G\notin [X]^7$ then replace it by any $G'\in [X]^7$ such that $G\subseteq G'$. By hypothesis, this means $\{x,y,z\}\in \Hom(\psi_G)$, and by (ii), this in turn implies that $\{x,y,z\} \in \Hom(\psi)$. On the other hand, if $\{x,y,z\}\in \Hom(\psi)$, then a completely analogous argument shows that $\{x,y,z\} \in \Hom(\varphi)$.
\end{itemize}
\bigskip 

As $D_1(\varphi+\psi)\neq \emptyset$ and $\varphi H\psi$, we conclude that  $\psi$ is a non trivial reconstruction of $\varphi$. Thus $\varphi\not\in \mathcal{R}$.
 \end{description}
\endproof

\begin{cor}
Let $\varphi$ be a coloring on $X$  where $|X|\geq R(7)$. Then $\varphi\in \mathcal{R}$ iff for every $F\in [X]^4$ one of the following holds
\begin{itemize}
\item[(i)] $\varphi\restriction F\in \mathcal{R}$.
    
\item[(ii)] $\varphi\restriction F\not\in \mathcal{R}$ and  there is  $G\supseteq F$ with $|G|\leq 7$ such that: Either  $\varphi\restriction [G]^2\in \mathcal{R}$ or  for all non-trivial reconstruction $\psi_F$ of $\varphi\restriction F$, there is a non-trivial reconstruction $\psi_G$  of $\varphi\restriction [G]^2$ such that 
$D_1(\varphi\restriction [G]^2 +\psi_G)\neq D_1(\varphi\restriction [F]^2+\psi_F)$.

\end{itemize}
\end{cor}

\begin{rem}The hypothesis in theorem \ref{prop:RnotSRcol}, that there is 
a coloring $\psi_G$ on $G$ such that $\varphi\restriction [G]^2 H \psi_G$ and $D_1(\varphi\restriction [G]^2 +\psi_G)=D_1(\varphi\restriction [F]^2+\psi_F)$, is stronger than just requiring that $\psi_G$ is a non-trivial reconstruction of $\varphi\restriction [G]^2$; as we illustrate in the example that follows.
\end{rem}

\begin{ex}
\label{ex:alpha}
{\em 
We present a coloring $\alpha\in \mathcal{R}$ which is not $\mathcal{SR}$, in particular, it does not have property $E_i$, for $i=0,1$.  

Let $\alpha$ be the  coloring on $\mathbb{N}$ defined by:
    \begin{equation}\label{eq:alpha}
        \begin{alignedat}{2}
            \alpha\{0,1\}&= \alpha\{0,2\} = 1-\alpha\{1,2\},\\
            \alpha\{n,n+1\} &= \alpha\{0,n+1\} = 1-\alpha\{0,n\} \text{ for every $n\ge 2$},\\
            \alpha\{k,n\} &=1-\alpha\{0,k\} \text{ for every $1\leq k <n$}.
        \end{alignedat}
    \end{equation}

    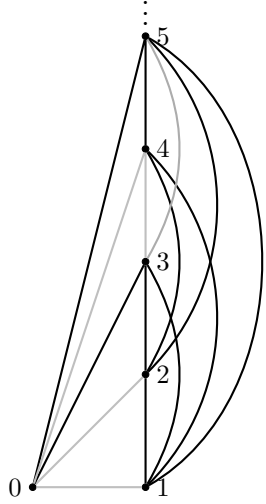
\begin{figure}[h]
        \centering        \begin{tikzpicture}[scale=1.5]
\node[left] (0) at (0,0) {\footnotesize$0$};
\node[right] (1) at (1,0) {\footnotesize$1$};
\node[right] (2) at (1,1) {\footnotesize$2$};
\node[right] (3) at (1,2) {\footnotesize$3$};
\node[right] (4) at (1,3) {\footnotesize $4$};
\node[right] (5) at (1,4) {\footnotesize $5$};
\node[above] (l) at (1,4) {\small $\vdots$};
\bgroup
\tikzstyle{every node}=[circle, draw, fill=black,inner sep=0pt, minimum width=2.5pt]

\node (0) at (0,0) {};
\node (1) at (1,0) {};
\node (2) at (1,1) {};
\node (3) at (1,2) {};
\node (4) at (1,3) {};
\node (5) at (1,4) {};

\draw[gray!50,thick] (2) -- (0) -- (1);
\draw[thick] (1) -- (2);
\draw[thick] (2) -- (3) -- (0) --(5) -- (4);
\draw[thick, gray!50] (0) -- (4) -- (3);
\draw[thick] (1) to[out=60, in=-60] (3);

\draw[thick] (1) to[out=45, in=-45] (4);

\draw[thick] (1) to[out=30, in=-30] (5);

\draw[thick] (2) to[out=60, in=-60] (4);

\draw[thick] (2) to[out=45, in=-45] (5);

\draw[thick, gray!65] (3) to[out=60, in=-60] (5);

\egroup

\end{tikzpicture}
        \caption{It shows the restriction $\alpha \restriction [\{0,1,2,3,4,5\}]^2$. The edges $\{k,n\}$ in the third equation are not shown, they are all the same color as $\{k,k+1\}$.} 
        \label{fig:enter-label}
    \end{figure}

\noindent Observe that, for any $n$, $\{0,n+2\}$ is a critical pair for the restriction of $\alpha$ to $\{0, 1,\ldots, n+1\}$ but it is no longer a critical pair on $\{0, 1,\ldots, n+2\}$. 

To see that $\alpha\not\in \mathcal{R}$, it suffices to show that $\alpha$ has neither critical pair nor critical cycles. 

First note that any triple $\{1,2,z\}\in [\N]^3$ is a homogeneous set for $\alpha$. By the first and third equation in \eqref{eq:alpha}, we have 
\[\begin{alignedat}{3}
    \alpha\{1,2\} &= 1-\alpha\{0,1\} &= \alpha\{1,n\} \text{ for every $n>2$,}\\
    &=1-\alpha\{0,2\} &= \alpha\{2,n\} \text{ for every $n> 2$.}
\end{alignedat}\]
Thus,  $\{1,2\}$, $\{1,n\}$ and $\{2,n\}$,  for all $n>2$, are neither  critical pairs nor belongs to a critical cycle for $\alpha$. 

For any $\{x,y\}\in [X]^2$ define the set 
\[
B_{\{x,y\}} = \{ z\in X\setminus\{x,y\}\,:\, \alpha\{x,z\}=\alpha\{y,z\} \}.
\]
We have observed after equation \eqref{eq:criticos} that $|B_{\{x,y\}}|\leq 1$ whenever $\{x,y\}$ is either a critical pair or it belongs to a critical cycle for $\alpha$. 
We will show that $|B_{\{x,y\}}|>1$ for every $\{x,y\}\in [\N]^2$ such that $\{x,y\}\cap\{1,2\}=\emptyset$ and for $\{x,y\}$ equal to $\{0,1\}$  or $\{0,2\}$. Thus all those pairs are neither critical nor belong to a critical cycle. 

Let  $\{x,y\}\in [\N \setminus \{1,2\}]^2$, we consider two cases. 

\medskip 

\noindent \textbf{Case 1:} If $x=0$ then it is enough to note that, by application of the second equation in \eqref{eq:alpha}, $y+1$ and $y+3$ are elements of $B_{\{x,y\}}$. So $|B_{\{x,y\}}|>1$.

\bigskip 

\noindent\textbf{Case 2:} Suppose $x\neq 0$. Since both $x$ and $y$ are greater than 2, by the third equation in \eqref{eq:alpha} we have $1,2\in B_{\{x,y\}}$. So again $|B_{\{x,y\}}|>1$.

\bigskip 

Finally, we show that $\alpha$ does not have the property $E_i$ for any $i\in \{0,1\}$. 
Let $F=\{1,2,3\}$. By the second equation $\alpha\{0,3\}=1-\alpha\{0,2\}$. Then the third equation in \eqref{eq:alpha} gives that for every $n>3$, $\alpha\{3,n\}=1-\alpha\{2,n\}$. We have proven that for every $n \in \N\setminus\{1,2,3\}$, $\alpha\{3,n\}\neq\alpha\{2,n\}$. 
}
\end{ex}


\begin{thebibliography}{100}
\bibliographystyle{plain}


\bibitem{bondy1991}
Bondy, J.~A.
\newblock {A Graph Reconstructor's Manual}. In {\em Surveys in Combinatorics},
\newblock London Mathematical Society Lecture Note Series, page 221--252. Cambridge University Press, 1991.

\bibitem{Dammak2019}
Dammak, J. and Si Kaddour, H.
\newblock {$(-1)$}-hypomorphic graphs with the same 3-element homogeneous subsets.
\newblock {\em Graphs and Combinatorics}, 35(2):427--436, 2019.

\bibitem{Pouzet2013}
Dammak, J., Lopez, G., Pouzet, M. and Si Kaddour, H.
\newblock Boolean sum of graphs and reconstruction up to complementation.
\newblock {\em Advances in Pure and Applied Mathematics.}, 4(3):315--349, 2013.

\bibitem{clari-uzca2023}
Pi\~{n}a, C. and Uzc\'ategui, C.
\newblock Reconstruction of a coloring from its homogeneous sets.
\newblock {\em Graphs and Combinatorics}, 39(6), 2022.


\bibitem{Pouzetetal2011}
Pouzet, M.,  Si Kaddour, H.  and  Trotignon, N.
\newblock Claw-freeness, 3-homogeneous subsets of a graph and a reconstruction
  problem.
\newblock {\em Contrib. Discrete Math.}, 6(1):86--97, 2011.

\bibitem{Pouzetetal2016}
{Pouzet, M. and Si Kaddour, H}.
\newblock {Isomorphy up to complementation}.
\newblock {\em Journal of Combinatorics}, 7(2-3): 285--305, 2016.


\bibitem{Radz2017}
  Radziszowski, S. 
  \newblock Small Ramsey numbers.
  \newblock {\em Electron. J. Comb.} http://www.
combinatorics.org/files/Surveys/ds1/ds1v15-2017.pdf, 2017.
\end{thebibliography}
\end{document}